\documentclass{article}
\usepackage[utf8]{inputenc}
\usepackage{amssymb}
\usepackage{amsfonts}
\usepackage{mathrsfs}
\usepackage{amsmath}
\usepackage{graphicx}
\usepackage{dsfont}
\usepackage{float}
\usepackage{diagbox} 

\usepackage[utf8]{inputenc}
\usepackage{graphics}
\usepackage{graphicx}
\usepackage[all]{xy}
\usepackage[margin=1.4in,marginparwidth=1in,marginparsep=.05in]{geometry}
\PassOptionsToPackage{hyphens}{url}
\usepackage{tikz-cd}
\usepackage{titlesec}
\usepackage{braket}
\usepackage{xcolor,graphicx}
\usepackage{verbatim}
\usepackage{float}
\usepackage{mathrsfs}
\usepackage{extarrows}
\usepackage{tikz}
\usetikzlibrary{shapes,arrows,positioning}
\usepackage{mathrsfs}
\usepackage{mathtools}

\usepackage{subfig}
\usepackage{hyperref}

\usepackage{tikz}
\usetikzlibrary{arrows}

\makeatletter
\newcommand*{\rom}[1]{\expandafter\@slowromancap\romannumeral #1@}

\titlespacing*{\section}{0pt}{\baselineskip}{\baselineskip}
\titleformat{\subsection}{\normalfont\bfseries}{\thesubsection.}{1em}{}
\titleformat{\subsubsection}{\normalfont}{\thesubsubsection.}{1em}{\itshape}
\newtheorem{theorem}{Theorem}[section]

\newtheorem{lemma}[theorem]{Lemma}

\newtheorem{remark}[theorem]{Remark}

\numberwithin{equation}{section}

\newenvironment{proof}{\paragraph{Proof:}}{\hfill$\square$}

\newcommand{\bD}{\mathbf D}

\newcommand{\bJ}{\mathbf J}
\newcommand{\bP}{\mathbf P}

\newcommand{\bQ}{\mathbf Q}

\newcommand{\bc}{\mathbf c}
\newcommand{\bg}{\mathbf g}

\newcommand{\bn}{\mathbf n}
\newcommand{\be}{\mathbf e}

\newcommand{\bu}{\mathbf u}

\newcommand{\bv}{\mathbf v}
\newcommand{\bw}{\mathbf w}

\newcommand{\bx}{\mathbf x}

\newcommand{\bbf}{\mathbf f}
\newcommand{\A}{\mathcal A}

\newcommand{\T}{\mathcal T}

\newcommand{\cE}{\mathcal E}

\newcommand{\OGamma}{\Omega^\Gamma_h}

\renewcommand{\div}{\textrm{div}\ \!}

\DeclareGraphicsExtensions{.pdf,.eps,.ps,.eps.gz,.ps.gz,.eps.Y}

\newcommand{\la}{\left\langle}
\newcommand{\ra}{\right\rangle}

\newcommand{\bsigma}{\boldsymbol{\sigma}}

\def\cl {\nonumber \\}
\def\el {\nonumber }

\colorlet{Mycolor1}{green!10!orange!90!}

\newcommand{\compositenorm}[1]{{\left\vert\kern-0.25ex\left\vert\kern-0.25ex\left\vert #1
    \right\vert\kern-0.25ex\right\vert\kern-0.25ex\right\vert}}

\usepackage{accents}

\title{An unfitted finite element method for two-phase Stokes problems with slip between phases}
\author{M. Olshanskii$^{1}$, A. Quaini$^{1}$ and Q. Sun$^{1}$}

\begin{document}

\maketitle

\begin{center}
\noindent $^{1}$Department of Mathematics, University of Houston, 3551 Cullen Blvd, Houston TX 77204\\
\texttt{molshan@math.uh.edu; aquaini@central.uh.edu; qsun5@uh.edu}
\end{center}

\vskip .5cm
{\bf Abstract}

We present an isoparametric unfitted finite element approach
of the CutFEM or Nitsche-XFEM family for the simulation of
two-phase Stokes problems with slip between phases.
For the unfitted generalized Taylor--Hood finite element pair
$\bP_{k+1}-P_k$, $k\ge1$, we show an inf-sup stability property with a stability constant that is independent of the viscosity ratio,
slip coefficient, position of the interface with respect to the background mesh and, of course, mesh size.
In addition, we prove stability and optimal error estimates that follow from this inf-sup property.
We provide numerical results in two and three
dimensions to corroborate the theoretical findings and demonstrate the robustness
of our approach with respect to the contrast in viscosity, slip coefficient value, and position of the interface
relative to the fixed computational mesh.


\section{Introduction}\label{sec:intro}

The finite element approximation of two-phase problems involving immiscible fluids features several challenging aspects.
The first challenge is the presence of a sharp interface between the two phases, that might move and undergo
topological changes. A second critical aspect is the presence of surface tension forces that create a jump in the
pressure field at the interface. In addition, if one accounts for slip between phases \cite{HAPANOWICZ2008559},
a jump in the velocity field at the interface needs to be captured as well.
Finally, lack of robustness may arise when there is a high contrast in fluid densities and
viscosities. Tackling all of these challenges has motivated a large body of literature.

One possible way to categorize numerical methods proposed in the literature is to distinguish
between \emph{diffusive interface} and \emph{sharp interface} approaches.
Phase field methods (e.g., \cite{Anderson1998,1999JCoPh}) belong to the first category, while level set methods (e.g., \cite{SUSSMAN1994146}),
and conservative level set methods (e.g., \cite{OLSSON2005225}) belong to the second.
Diffusive interface methods introduce a smoothing region around the interface between the two phases to vary smoothly,
instead of sharply, from one phase to the other and usually apply the surface tension forces
over the entire smoothing region. The major limitation of diffusive interface methods lies
in the need to resolve the smoothing region with an adequate number of elements, which
results in high computational costs. Sharp interface methods require less elements to resolve
the interface between phases. Thus, we will restrict our attention to sharp interface approaches, which
can be further divided into \emph{geometrically fitted} and \emph{unfitted} methods.

In fitted methods, the discretization mesh is fitted to the computational interface.
Perhaps, Arbitrary Lagrangian Eulerian (ALE) methods \cite{Donea2004} are the best known fitted methods.
In case of a moving interface, ALE methods deform the mesh to track the interface.
While ALE methods are known to be very robust for small interface displacement,
complex re-meshing procedures are needed for large deformations and topological changes.
Certain variations of the method, like the extended ALE \cite{Basting2013c,BASTING2017312},
successfully deal with large interface displacement while keeping the same mesh connectivity.
The price to pay for such improvement is a higher computational cost.
Unfitted methods allow the sharp interface to cut through the elements of a
fixed background grid. Their main advantage is the relative ease of handling time-dependent
domains, implicitly defined interfaces, and problems with strong geometric
deformations \cite{Bordas2018}.
The immersed finite element method (e.g., \cite{ADJERID2015170}) and front-tracking methods (e.g., \cite{osti_7310568})
are examples of unfitted approaches. Applied in the finite element framework, these methods require
an enrichment of the elements intersected by the interface  in order to capture jumps and kinks in the solution.
One complex aspect of these methods is the need for tailored stabilization.
Popular unfitted methods that embed discontinuities in finite element solvers are
XFEM \cite{Moes1999} and CutFEM \cite{cutFEM}. XFEM enriches the
finite element shape functions by the Partition-of-Unity method. To learn more about XFEM
applied to two-phase flow problems, we refer the reader to \cite{Belytschko03,Fries2009,Gross04,Kirchhart_Gross2016,SAUERLAND201341}. 
CutFEM is a variation of XFEM, also called Nitsche-XFEM \cite{Hansbo02}.
CutFEM uses overlapping fictitious domains in combination with ghost
penalty stabilization \cite{B10} to enrich and stabilize the solution.
See \cite{CLAUS2019185,FRACHON201977,HANSBO201490,He_Song2019,Massing_Larson2014,Wang_Chen2019}
for the application of CutFEM or Nitsche-XFEM to approximate two-phase flows.
Finally, recently proposed unfitted methods are a hybrid high-order method \cite{burman_delay2020}
and an enriched finite element/level-set method \cite{HASHEMI2020113277}.

In this paper, we study an isoparametric unfitted finite element approach
of the CutFEM or Nitsche-XFEM family for the simulation of
two-phase Stokes problems with slip between phases. All the numerical works cited above
consider the homogeneous model of two-phase flow, i.e.~no slip is assumed between the phases.
This assumption is appropriate in three cases: one of the phases has
a relatively small volume, one phase forms drops of minute size, or one phase
(representing the continuous medium in which droplets are immersed) has high speed
\cite{HAPANOWICZ2008559}. In all other cases, slip between the phases has to be accounted for.
In fact, experimentally it is observed that the velocity of the two phases can be significantly different, also
depending on the flow pattern (e.g., plug flow, annular flow, bubble flow, stratified flow, slug flow, churn flow)
\cite{KERMANI201113235}.
A  variation of our unfitted approach has been analyzed for the homogeneous
two-phase Stokes problem in \cite{caceres2019new}, where robust estimates were proved for  individual terms of
the Cauchy stress tensor. In the present paper, the analysis is done in the energy norm,  allowing a possible slip between phases. 
In particular, we show an inf-sup stability property of the unfitted generalized Taylor--Hood finite element pair
$\bP_{k+1}-P_k$, $k\ge1$, with a stability constant that is independent of the viscosity ratio, slip coefficient, position of the
interface with respect to the background mesh, and of course mesh size. 
This inf-sup property implies stability and optimal error estimates for the unfitted finite element
method under consideration, which are also shown.
For more details on the isoparametric unfitted finite element, we refer to
\cite{lehrenfeld2016high,Lehrenfeld2017,lehrenfeld2018analysis}.

Two-phase flow problems with high contrast for the viscosity are known to be especially challenging.
While some authors test different viscosity ratios but do not comment on the effects of high contrast on the numerics \cite{CLAUS2019185,HASHEMI2020113277,WANG2015820}, others
show or prove that their method is robust for all viscosity
ratios \cite{He_Song2019,burman_delay2020,Kirchhart_Gross2016,olshanskii2006analysis,Wang_Chen2019}.
In other cases, numerical parameters, like the penalty parameters,
are adjusted to take into account large differences in the viscosity \cite{FRACHON201977}.
Through analysis and a series of numerical tests in two and three dimensions, we demonstrate
that our approach is robust not only with respect to the contrast in viscosity, but also
with respect to the slip coefficient value and the position of the interface relative to the fixed computational mesh.

For all the simulations in this paper, we have used NGsolve \cite{NGSolve,Gangl2020},
a high performance multiphysics finite element software with a  Python interface,
and add-on library ngsxfem \cite{ngsxfem}, which enables the use of unfitted finite element technologies.

The remainder of the paper is organizes as follows.
In Section \ref{sec:t_dep}, we introduce the strong and weak formulations of
the two-phase Stokes problem with slip between phases, together with the finite element discretization.
We present a stability result in Sec.~\ref{sec:stab}, while in Sec.~\ref{sec:error} we
prove optimal order convergence for the proposed unfitted finite element approach.
Numerical results in 2 and 3 dimensions are shown in Sec.~\ref{num_res}.
Concluding remarks are provided in Sec.~\ref{ref:concl}.

\section{Problem definition}\label{sec:t_dep}

We consider a fixed domain $\Omega\subset \mathbb{R}^d$, with $d = 2,3$, filled with two immiscible, viscous,
and incompressible fluids separated by an interface $\Gamma$. In this study, we assume $\Gamma$
does not evolve with time although our approach is designed to handle interface evolution.
We assume that $\Gamma$ is closed and sufficiently smooth.
Interface $\Gamma$ separates $\Omega$ into
two subdomains (phases) $\Omega^+$ and $\Omega^-=\Omega\setminus \overline{\Omega^{+}}$.
We assume $\Omega^-$ to be completely internal, i.e.~$\partial\Omega^-\cap \partial\Omega = \emptyset$.
See Fig.~\ref{fig:Geom}.
Let $\bn^{\pm}$ be the outward unit normal for $\Omega^{\pm}$
and $\bn$ the outward pointing unit normal on $\Gamma$.
It holds that $\bn^- = \bn$ and $\bn^+ = - \bn$ at $\Gamma$.

\begin{figure}[H]\label{Geom}
    \centering
    \includegraphics[width=.4\textwidth]{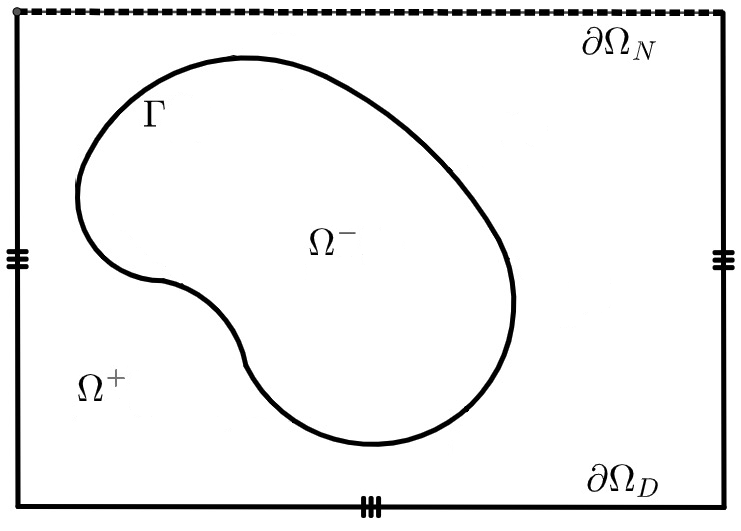}
    \caption{Illustration of a domain $\Omega$ in $\mathbb{R}^2$. On part of the boundary (dashed line) a Neumann boundary
    condition is imposed, while on the remaining part of the boundary (solid line with three bars) a Dirichlet boundary condition
    is enforced.
    \label{fig:Geom}}
\end{figure}

Let $\bu^{\pm}:\Omega^{\pm}\to \mathbb{R}^d$ and $p^{\pm}:\Omega^{\pm}\to \mathbb{R}$
denote the fluid velocity and pressure, respectively.
We assume that the motion of the fluids occupying subdomains $\Omega^{\pm}$ can be modeled by
the Stokes equations
\begin{align}
   - \nabla\cdot\bsigma^{\pm} & = \bbf^{\pm} &\text{ in }\Omega^{\pm}, \label{eq:Stokes1} \\
    \nabla\cdot{}\bu^{\pm}&=0 &\text{ in }\Omega^{\pm}, \label{eq:Stokes2}
\end{align}
endowed with boundary conditions
\begin{align}
    \bu^+&=\bg, &\text{ on }\partial\Omega_D, \label{eq:bcD} \\
    \bsigma^+\bn^+ &= \bg_N &\text{ on }\partial\Omega_N. \label{eq:bcN}
\end{align}
Here, $\overline{\partial\Omega_D}\cup\overline{\partial\Omega_N}=\overline{\partial\Omega}$ and $\partial\Omega_D \cap\partial\Omega_N=\emptyset$. See Fig.~\ref{fig:Geom}.
In \eqref{eq:Stokes1},
$\bbf^{\pm}$ are external the body forces and $\bsigma^{\pm}$ are the Cauchy stress tensors.
For Newtonian fluids, the Cauchy stress tensor has the following expression:
\begin{align}
\bsigma^{\pm}= -p^{\pm} \mathbf{I} +2\mu_{\pm} \bD(\bu^{\pm}), \quad \bD(\bu^{\pm}) = \frac{1}{2} (\nabla \bu^{\pm} + (\nabla\bu^{\pm})^T) \text{ in }\Omega^{\pm}, \el
\end{align}
where constants $\mu_{\pm}$ represent the fluid dynamic viscosities.
Finally, $\bg$ and $\bg_N$ in \eqref{eq:bcD} and \eqref{eq:bcN} are given.

Subproblems \eqref{eq:Stokes1}-\eqref{eq:Stokes2} are coupled at the interface $\Gamma$.
The conservation of mass requires the balance of normal fluxes on $\Gamma$:
 \begin{align}
    \bu^+\cdot \bn&= \bu^-\cdot \bn\qquad  \text{ on }\Gamma. \label{eq:scc1}
 \end{align}
This is the first coupling condition.
We are interested in modelling slip with friction between the two phases.
Thus, we consider the following additional coupling conditions:
\begin{align}
\bP{\bsigma^+\bn}&=f(\bP\bu^+-\bP\bu^-)  &&\text{ on }\Gamma,  \label{eq:scc2}\\
\bP{\bsigma^-\bn}&=-f(\bP\bu^--\bP\bu^+)  &&\text{ on }\Gamma,  \label{eq:scc3}
\end{align}
where $f$ is a constant that can be seen as a slip coefficient and
$\bP=\bP(\bx)=I-\bn(\bx)\bn(\bx)^T$ for $\bx \in \Gamma$ is the orthogonal projection onto the tangent plane.
Finally, the jump of the normal stress across $\Gamma$ is given by:
\begin{align}
[\bn^T\bsigma\bn]^-_+&=\sigma \kappa \qquad\text{ on }\Gamma, \label{eq:scc4}
\end{align}
where $\sigma$ is the surface tension coefficient and $\kappa$ is the double mean curvature of the interface.

Since the boundary conditions on $\partial\Omega$ do not affect the subsequent discussion,
from now on we will consider that a Dirichlet condition \eqref{eq:bcD} is imposed on the entire boundary.
This will simplify the presentation of the fully discrete problem.

\subsection{Variational formulation}
The purpose of this section is to derive the variational formulation of coupled problem \eqref{eq:Stokes1}--\eqref{eq:scc4}.
Let us introduce some standard notation. The space of functions whose square is integrable in a domain $\omega$
is denoted by $L^2(\omega)$.
With $L^2_0(\omega)$, we denote the space of functions in $L^2(\omega)$ with zero mean value over $\omega$.
The space of functions whose distributional derivatives of order up to $m \geq 0$  (integer)
belong to $L^2(\omega)$ is denoted by $H^m(\omega)$.
The space of vector-valued functions with components in $L^2(\omega)$ is denoted with $L^2(\omega)^d$.
$H^1(\div,\omega)$ is the space of functions in $L^2(\omega)$ with divergence in $L^2(\omega)$.
Moreover, we introduce the following functional spaces:
\begin{align}
V^- &= H^1(\Omega^-)^d, ~V^+ =\{\bu \in H^1(\Omega^+)^d, \bu{\big|_{\partial \Omega_D}}=\bg \}, ~V^+_0 =\{\bu \in H^1(\Omega^+)^d, \bu{\big|_{\partial \Omega_D}}=\boldsymbol{0} \}, \cl
V^\pm & = \{ \bu = (\bu^-,\bu^+) \in V^- \times V^+ , \bu^- \cdot \bn = \bu^+ \cdot \bn~\text{on}~\Gamma \}, \cl
V^\pm_0 & = \{ \bu = (\bu^-,\bu^+) \in V^- \times V^+_0 , \bu^- \cdot \bn = \bu^+ \cdot \bn~\text{on}~\Gamma \}, \cl
Q^\pm & = \{ p = (p^-,p^+) \in L^2(\Omega^-)\times L^2(\Omega^+) \}. \el
\end{align}
Notice that space $V^\pm$ can be also characterized as $(V^- \times V^+)\cap H^1(\div,\Omega)$.
We use $( \cdot, \cdot)_\omega$ and $\langle, \rangle_\omega$ to denote the $L^2$ product and the duality pairing, respectively.

The integral formulation of the problem \eqref{eq:Stokes1}-\eqref{eq:scc4} reads:
Find $(\bu, p) \in V^\pm \times L^2(\Omega)/\mathbb{R}$ such that
\begin{align}
&- (p^-,\nabla \cdot \bv^-)_{\Omega^{-}} - (p^+,\nabla \cdot \bv^+)_{\Omega^{+}} + 2(\mu_-\bD(\bu^-),\bD(\bv^-))_{\Omega^{-}}  + 2(\mu_+\bD(\bu^+),\bD(\bv^+))_{\Omega^{+}}  \cl
&\quad \quad  + \langle f(\bP\bu^--\bP\bu^+), \bP\bv^- \rangle_\Gamma + \langle f(\bP\bu^+-\bP\bu^-), \bP\bv^+ \rangle_\Gamma   \cl
&\quad \quad = (\bbf^-,\bv^-)_{\Omega^{-}} + (\bbf^+,\bv^+)_{\Omega^{+}} + \langle \sigma \kappa , \bv^- \cdot \bn \rangle_\Gamma \label{eq:swaek-1} \\
&(\nabla\cdot\bu^{-}, q^-)_{\Omega^{-}}+ (\nabla\cdot\bu^{+}, q^+)_{\Omega^{+}}=0 \label{eq:swaek-2}
\end{align}
for all $(\bv, q) \in V^{\pm}_0 \times Q^{\pm}$. The interface terms in \eqref{eq:swaek-1} have been obtained using
coupling conditions \eqref{eq:scc2}, \eqref{eq:scc3}, and \eqref{eq:scc4} as follows:
\begin{align}
-\langle \bsigma^- \bn, \bv^- \rangle_\Gamma  + \langle \bsigma^+ \bn, \bv^+\rangle_\Gamma &=
-\langle \bP\bsigma^- \bn, \bP\bv^- \rangle_\Gamma + \langle \bP\bsigma^+ \bn, \bP\bv^+\rangle_\Gamma - \langle [\bn^T\bsigma \bn]^-_+, \bv^- \cdot \bn \rangle_\Gamma \cl
& = \langle f(\bP\bu^--\bP\bu^+), \bP\bv^- \rangle_\Gamma + \langle f(\bP\bu^+-\bP\bu^-), \bP\bv^+ \rangle_\Gamma \cl
& \quad  -\langle \sigma \kappa , \bv^- \cdot \bn \rangle_\Gamma. \el
\end{align}

Notice that problem \eqref{eq:swaek-1}-\eqref{eq:swaek-2}
can be rewritten as: Find $(\bu, p) \in V^\pm \times L^2(\Omega)/\mathbb{R}$ such that
\begin{equation}\label{sweak}
\left\{
\begin{split}
   a(\bu,\bv)+b(\bv, p)&= r(\bv) \\
   b(\bu,q)&=0
\end{split}
\right.
\end{equation}
for all $(\bv, q) \in V_0^{\pm} \times Q^\pm$, where
\begin{align}
    a(\bu,\bv)=& 2(\mu_-\bD(\bu^-),\bD(\bv^-))_{\Omega^{-}}  + 2(\mu_+\bD(\bu^+),\bD(\bv^+))_{\Omega^{+}} + \langle f(\bP\bu^--\bP\bu^+), \bP\bv^- - \bP\bv^+\rangle_\Gamma, \cl
    b(\bv,p)=&- (p^-,\nabla \cdot \bv^-)_{\Omega^{-}} - (p^+,\nabla \cdot \bv^+)_{\Omega^{+}},\cl
    r(\bv)=&(\bbf^-,\bv^-)_{\Omega^{-}} + (\bbf^+,\bv^+)_{\Omega^{+}}+ \langle \sigma \kappa , \bv^- \cdot \bn \rangle_\Gamma. \el
\end{align}

\subsection{Finite element discretization}\label{sec:FE}

We consider a family of shape regular triangulations $\{\T_h\}_{h >0}$ of $\Omega$.
We adopt the convention that the elements $T$ and edges $e$ are open
sets and use the over-line symbol to refer to their closure.
Let $h_T$ denote the diameter of element $T \in \T_h$
and $h_e$ the diameter of edge $e$.
The set of elements intersecting $\Omega^\pm$ and the set of elements
having a nonzero intersection with $\Gamma$ are
\begin{align}
\T^\pm_h = \{ T \in \T_h:T \cap \Omega^\pm \neq \emptyset \}, \quad
\T_h^\Gamma= \{ T \in \T_h:\overline{T} \cap \Gamma \neq \emptyset \},
\end{align}
respectively.
We assume $\{\T_h^\Gamma\}$ to be quasi-uniform.
However, in practice adaptive mesh refinement is possible.
The domain formed by all tetrahedra in $\T_h^\Gamma$ is denoted by $\OGamma:=\text{int}(\cup_{T \in \T_h^\Gamma} \overline{T})$.
We define the $h$-dependent domains:
\begin{align}
\Omega_h^\pm = \text{int}\left(\cup_{T \in \T_h^\pm} \overline{T}\right)
\end{align}
and the set of faces of $\T_h^\Gamma$ restricted to the interior of $\Omega_h^\pm$:
\begin{align}
\mathcal{E}_h^{\Gamma,\pm} = \{ e = \text{int}(\partial T_1 \cap \partial T_2):T_1, T_2 \in \T_h^\pm~\text{and}~T_1 \cap \Gamma \neq \emptyset~\text{or}~ T_2 \cap \Gamma \neq \emptyset \}.
\end{align}

For the space discretization of the bulk fluid problems, we restrict our attention to inf-sup stable
finite element pair $\bP_{k+1} - P_k$, $k \geq 1$, i.e.~Taylor-Hood elements. Specifically,
we consider the spaces of continuous finite element pressures given by:
\begin{align}
Q_h^- = \{ p \in C(\Omega_h^-): q|_T \in P_k(T)~\forall T \in \T^-_h \}.
\end{align}
Space $Q_h^+$ is defined analogously. Our pressure space is given by:
\begin{align}
Q^\pm_h = \{ p = (p^-, p^+) \in Q_h^- \times Q_h^+\,:\,\int_{\Omega^-}\mu_{-}^{-1}p^-+\int_{\Omega^+}\mu_{+}^{-1}p^+=0 \} . \el
\end{align}
Let
\begin{align}
V_h^- &= \{ \bu \in C(\Omega_h^-)^d: \bu|_T \in \bP_{k+1}(T)~\forall T \in \T^-_h \}. 
\end{align}
with the analogous definition for $V_h^+$. Our velocity spaces are given by:
\begin{align}
V^\pm_h = \{ \bu = (\bu^-, \bu^+) \in V_h^- \times V_h^+ \} \el
\end{align}
and $V_{0,h}^{\pm} $, a subspace of $V^\pm_h$ with vector functions $\bu^+$ vanishing on $\partial\Omega$. 
All above constructions and spaces readily carry over to tessellations of $\Omega$ into squares or cubes and using  $\bQ_{k+1} - Q_k$ elements.

Functions in $Q^\pm_h$ and $V^\pm_h$ and their derivatives are multivalued in $\OGamma$, the overlap of $\Omega_h^-$
and $\Omega_h^+$. The jump of a multivalued function over the
interface is defined as the difference of components coming from $\Omega_h^-$
and $\Omega_h^+$, i.e.~$[ \bu] = \bu^- - \bu^+$ on $\Gamma$. Note that this is the
jump that we have previously denoted with $[ \cdot]^-_+$.
We are now using $[ \cdot] $ to simplify the notation.
Moreover, we define the following averages:
\begin{align}
\{ \bu \} = \alpha \bu^+ + \beta \bu^-, \label{curly_av} \\
\langle \bu \rangle = \beta \bu^+ + \alpha \bu^-, \label{angle_av}
\end{align}
where $\alpha$ and $\beta$ are weights to be chosen such that $\alpha+\beta = 1$, $0 \leq \alpha, \beta \leq 1$.
For example, in \cite{CLAUS2019185} the setting $\alpha = \mu_-/(\mu_+ + \mu_-)$ and $\beta = \mu_+/(\mu_+ + \mu_-)$
is suggested. In \cite{caceres2019new}, the authors choose $\alpha=0$, $\beta=1$ if $\mu_{-} \le \mu_{+}$ and $\alpha=1$, $\beta=0$
otherwise.
Below, in \eqref{eq:a_1} and \eqref{eq:b_h} we will use
relationship:
\begin{align}
[ab] = [b] \{a\} + \la b \ra [a]. \label{eq:jump_av}
\end{align}

A discrete variational analogue of problem \eqref{sweak} reads:
Find $\{\bu_h, p_h\} \in V^\pm_h \times Q^\pm_h$ such that
\begin{equation}\label{FE_formulation}
\left\{
\begin{split}
   a_h(\bu_h,\bv_h)+b_h(\bv_h,p_h)&= r_h(\bv_h) \\
   b_h(\bu_h,q_h) - b_p(p_h, q_h)&=0
\end{split}
\right.
\end{equation}
for all $(\bv_h, q_h) \in V_{0,h}^{\pm} \times Q_h^\pm$.
We define all the bilinear forms in \eqref{FE_formulation} for all $\bu_h \in V_h^{\pm}$, $\bv_h\in V^\pm_{0,h}$, $p \in Q^\pm$. Let us start with form $a_h(\cdot, \cdot)$:
\begin{align}
    a_h(\bu_h,\bv_h)=& a_i(\bu_h,\bv_h) + a_n(\bu_h,\bv_h) + a_p(\bu_h,\bv_h), \label{eq:sa_h}
\end{align}
where we group together the terms that arise from the integration by parts of the divergence of the stress tensors:
\begin{align}
    a_i(\bu_h,\bv_h)=& 2(\mu_-\bD(\bu_h^-),\bD(\bv_h^-))_{\Omega^{-}}  + 2(\mu_+\bD(\bu_h^+),\bD(\bv_h^+))_{\Omega^{+}} + \langle f [\bP\bu_h], [\bP\bv_h]\rangle_\Gamma \cl
    &-2 \langle \{ \mu  \bn^T\bD(\bu_h) \bn \}, [ \bv_h \cdot \bn ] \rangle_\Gamma, \label{eq:a_1}
\end{align}
and the terms that enforce condition \eqref{eq:scc1} weakly using Nitsche's method
\begin{align}
a_n(\bu_h,\bv_h)= \sum_{T \in \mathcal{T}_h^\Gamma} \frac{\gamma}{h_T} \{ \mu \} \langle[\bu_h \cdot \bn], [\bv_h \cdot \bn]\rangle_{\Gamma\cap T} -
2 \langle \{ \mu \bn^T \bD(\bv_h) \bn \} ,  [\bu_h \cdot \bn] \rangle_\Gamma. \label{eq:a_n}
\end{align}
We recall that $h_T$ is the diameter of element $T \in \T_h$.
To define the penalty terms $a_p(\bu_h,\bv_h)$ we need  $\omega_e$, the facet patch for $e\in \mathcal{E}_h^{\Gamma, \pm}$ consisting of all $T\in \T_h$ sharing $e$. Then, we set
\begin{align}
a_p(\bu_h,\bv_h) &= {\mu_{-}}\bJ_h^-(\bu_h, \bv_h) + {\mu_{+}}\bJ_h^+(\bu_h, \bv_h), \cl
\bJ_h^\pm(\bu_h, \bv_h) &= \gamma^\pm_\bu \sum_{e \in \mathcal{E}_h^{\Gamma, \pm}}  \frac{1}{h^2_e} \int_{\omega_e} (\bu_1^e - \bu_2^e)\cdot (\bv_1^e - \bv_2^e) dx, \label{eq:a_j}
\end{align}
where $\bu_1^e$ is the componentwise canonical extension of a polynomial vector function $\bu^\pm_h$ from $T_1$ to $\mathbb{R}^d$, while  $\bu_2^e$ is the
canonical extension of  $\bu^\pm_h$ from $T_2$ to $\mathbb{R}^d$(and similarly for $\bv_1$, $\bv_2$).
We recall that $h_e$ is the diameter of facet $e\in \mathcal{E}_h^{\Gamma, \pm}$.
This version of the ghost penalty stabilization has been proposed in \cite{preussmaster}. In~\cite{lehrenfeld2018eulerian}, it was shown to be essentially equivalent to other popular ghost penalty stabilizations such as local projection stabilization~\cite{B10} and normal derivative jump stabilization~\cite{cutFEM}. In the context of the Stokes problem, this stabilization was recently used in  \cite{von2020unfitted}.
For the analysis in Sec.~\ref{sec:stab} and \ref{sec:error}, we also define $\bJ_h^\pm(\bu, \bv)$ for arbitrary smooth  functions $\bu,\bv$ in $\Omega_h^\pm$. In this case, we set $\bu_1 =  \left(\Pi_{T_1} \bu|_{T_1}\right)^e$, $\bu_2 = \left(\Pi_{T_2} \bu|_{T_2}\right)^e$,
where $\Pi_{T_i}$ is the $L^2(T_i)$-orthogonal projection into the space of degree $k+1$ polynomial vector functions on $T_i$.

The remaining terms coming from the integration by parts of the divergence of the stress tensors are contained in
\begin{align}
        b_h(\bv_h,p_h)=&- (p_h^-,\nabla \cdot \bv_h^-)_{\Omega^{-}} - (p_h^+,\nabla \cdot \bv_h^+)_{\Omega^{+}}
    + \la \{ p_h  \}, [ \bv_h \cdot \bn] \ra_\Gamma,\label{eq:b_h}
\end{align}
and the penalty terms are grouped together in
\begin{align}
b_p(p_h, q_h) &= \mu_{-}^{-1}J_h^-(p_h, q_h) + \mu_{+}^{-1} J_h^+(p_h, q_h), \cl
J_h^\pm(p_h, q_h) &= {\gamma_p^\pm} \sum_{e \in \mathcal{E}_h^{\Gamma, \pm}} \int_{\omega_e} (p_1^e - p_2^e) (q_1^e - q_2^e) dx, \label{eq:Jh_p}
\end{align}
where $p_1^e,p_2^e,q_1^e,q_2^e$ are canonical polynomial extensions as defined above.

Finally,
\begin{align}
     r_h(\bv_h)=&(\bbf_h^-,\bv_h^-)_{\Omega^{-}} + (\bbf_h^+,\bv_h^+)_{\Omega^{+}} + \langle \sigma \kappa , \la \bv_h \cdot \bn \ra \rangle_\Gamma. \el
\end{align}
We recall that some of the interface terms in $a_i(\cdot,\cdot)$ and $b_h(\cdot,\cdot)$
have been obtained using relationship \eqref{eq:jump_av} and interface conditions.

Parameters $\gamma^\pm_\bu$, $\gamma_p^\pm$ and $\gamma$ are all assumed to be independent
of $\mu_\pm$, $h$, and the position of $\Gamma$ against the underlying mesh.
Parameter $\gamma$
in \eqref{eq:a_n} needs to be large enough to provide the bilinear form $a_h(\cdot,\cdot)$
with coercivity.
Parameters $\gamma^\pm_\bu$, $\gamma_p^\pm$ can be tuned to improve
the numerical performance of the method.

\subsubsection{Numerical integration}\label{s:numint} It is not feasible to compute integrals entering the definition of the bilinear forms over cut elements and over $\Gamma$ for an arbitrary smooth $\Gamma$. We face the same problem if $\Gamma$ is given implicitly as a zero level of a piecewise polynomial function for polynomial degree greater than one. Piecewise linear approximation of $\Gamma$ on the given mesh and polygonal approximation of subdomains lead to second order geometric consistency error, which is suboptimal for Taylor--Hood elements.
To  ensure a geometric error of the same order or higher than the finite element (FE) approximation error, we define numerical quadrature rules on the given mesh using the isoparametric approach proposed in \cite{lehrenfeld2016high}.

In the isoparametric approach, one considers a smooth function $\phi$ such that  $\pm\phi>0$ in $\Omega^{\pm}$ and $|\nabla \phi|>0$ in a sufficiently wide strip around $\Gamma$.
Next, one defines  polygonal auxiliary domains $\Omega^{\pm}_1$  given by
$
\Omega^\pm_1:=\{\bx\in\mathbb{R}^d\,:\, \pm I_h^1(\phi)>0\},
$
where $I_h^1$ is the continuous piecewise \textit{linear}   interpolation of $\phi$ on $\T_h$.
Interface $\Gamma_1$ between $\Omega^+_1$ and $\Omega^{-}_1$ is then
$
\Gamma_1:=\{\bx \in\mathbb{R}^d\,:\, I_h^1(\phi)=0\}.
$
On  $\Omega^\pm_1$ and $\Gamma_1$ standard  quadrature rules can be applied elementwise.
Since using $\Omega^\pm_1$, $\Gamma_1$ alone limits the accuracy to second order,
one further constructs a transformation of the mesh in  $\T^\Gamma_h$
with the help of an explicit mapping $\Psi_h$ parameterized by a finite element function.
The mapping $\Psi_h$ is such that $\Gamma_1$ is mapped \textit{approximately} onto $\Gamma$;
see \cite{lehrenfeld2016high} for how $\Psi_h$ is constructed.
Then,  $\widetilde\Omega^\pm=\Psi_h(\Omega^\pm_1)$, $\widetilde \Gamma=\Psi_h(\Gamma_1)$ are high order
accurate approximations to the phases and interface which  have an explicit representation
so that the integration over $\widetilde\Omega^\pm$ and $\widetilde\Gamma$ can be done exactly.
The finite element spaces  have to be adapted correspondingly,
using the  explicit pullback mapping:  $\bv_h\circ\Psi_h^{-1}$.

\section{Stability}\label{sec:stab}

For the analysis in this and the next section, we assume that the integrals over cut elements in $\Omega^\pm$ are computed exactly.
In addition, we restrict our attention to the choice $\alpha = 0$ and $\beta =1$ for the averages in \eqref{curly_av}--\eqref{angle_av}, 
assuming $\mu_{-}\le\mu_+$.

The key for the stability analysis of the two-phase Stokes problem is an inf-sup stability
property of the unfitted generalized Taylor--Hood finite element pair,
which extends the classical LBB stability result for the standard $\bP_{k+1}-P_{k}$ Stokes element from~\cite{bercovier1979error}.
There is no similar stability result in the literature for $\bQ_{k+1}-Q_{k}$ unfitted elements.
However, we expect that the extension, and so the analysis below, can be carried over to these elements as well.

One is interested in  the  inf-sup inequality with a stability constant that is independent of the viscosity ratio,
position of $\Gamma$ with respect to the background mesh and, of course, mesh size $h$. 
The result is given in the following lemma.

\begin{lemma}\label{Lem1}
Denote by $ V_h$ the space of continuous $P_{k+1}$ finite element vector functions on $\Omega$,
$ V_h = \{ \bu \in C(\Omega)^d: \bu|_T \in \bP_{k+1}(T)~\forall T \in \T_h \}$.
There exists $h_0>0$ such that for all $h<h_0$  and any $q_h\in Q_h^{\pm}$ there exists $\bv_h\in  V_h$ such that it holds
\begin{equation}\label{eq:lem1}
\begin{split}
\mu_{-}^{-1}\| q_h^{-}\|_{{\Omega}_{h}^{-}}^2
  +\mu_{+}^{-1}\| q_h^{+}\|_{{\Omega}_{h}^{+}}^2
&\le (q_h^{-},\nabla\cdot \bv_h)_{{\Omega}^{-}}+(q_h^{+},\nabla\cdot \bv_h)_{{\Omega}^{+}} +c\,b_p(q_h, q_h)
\\
\|\mu^{\frac12}\nabla\bv_h\|_{\Omega}^2&\le
  C\, \left(\mu_{-}^{-1}\| q_h^{-}\|_{{\Omega}_{h}^{-}}^2
  +\mu_{+}^{-1}\| q_h^{+}\|_{{\Omega}_{h}^{+}}^2\right).
\end{split}
\end{equation}
with $h_0$ and two positive constants $c$ and $C$  independent of $q_h$, $\mu_\pm$, the position of $\Gamma$
in the background mesh and mesh size $h$.
\end{lemma}
\begin{proof}
Consider subdomains $\Omega_{h,i}^\pm\subset \Omega^\pm$ built of all strictly internal simplexes  in each phase:
\[
\overline{\Omega}_{h,i}^\pm  := \bigcup\{ \overline{T}\,: \, T \in \mathcal{T}_h,\quad T \subset \Omega^\pm \}.
\]
The following two results are central for the proof.
First, we have the uniform inf-sup inequalities in ${\Omega}_{h,i}^{-}$ and ${\Omega}_{h,i}^{+}$~\cite{guzman2018inf}:
there exist constants $C_{\pm}$ independent of the position of $\Gamma$ and $h$ such that
\begin{equation}\label{GO18}
0<C_{\pm}\le\inf_{q \in Q_h^{\pm}\cap L^2_0({\Omega}_{h,i}^{\pm})}~\sup_{\scriptsize
\begin{array}{c}
\bv \in  V_h \\ {\rm supp}(\bv)\subset{\Omega}_{h,i}^{\pm}
\end{array}}
\frac{(q,\nabla\cdot \bv)_{{\Omega}_{h,i}^{\pm}}}{\|\bv\|_{H^1({\Omega}_{h,i}^{\pm})}\|q\|_{{\Omega}_{h,i}^{\pm}}}.
\end{equation}
The above result can be equivalently formulated as follows: For any  $q \in Q_h^{\pm}\cap L^2_0({\Omega}_{h,i}^{\pm})$ there exist $\bv_h^\pm \in  V_h $ such that ${\rm supp}(\bv)\subset{\Omega}_{h,i}^{\pm}$ and
\begin{equation}\label{aux674}
\| q^\pm\|_{{\Omega}_{h,i}^{\pm}}^2
 = \left(q^\pm,\nabla\cdot \bv_h^\pm\right)_{{\Omega}^{\pm}_h },\quad
 \|\nabla\bv_h^\pm\|_\Omega\le C_{\pm}^{-1} \| q^\pm\|_{{\Omega}_{h,i}^{\pm}}.
\end{equation}
The second important results is the simple observation that the $L^2$ norm of $q_h$ in $\Omega_{h}^{\pm}$  can be controlled by the $L^2$ norm in  ${\Omega}_{h,i}^{\pm}$ plus the stabilization term in \eqref{eq:Jh_p} (see, \cite{lehrenfeld2018eulerian,preussmaster}):
\begin{equation}\label{Preuss}
\|q_h\|_{{\Omega}_{h}^{\pm}}^2 \le C\,(\|q_h\|_{{\Omega}_{h,i}^{\pm}}^2+ J_h^\pm(q_h, q_h)),
\end{equation}
with some constant $C$ independent of the position of $\Gamma$ and $h$.
We note that \eqref{Preuss} holds also for discontinuous finite elements.

Consider now
\[
q_\mu =\left\{
\begin{array}{r}
  ~\mu_{-} |\Omega^{-}|^{-1} \in Q_h^{-}  \\
  -\mu_{+} |\Omega^{+}|^{-1} \in Q_h^{+}.
\end{array}
\right.
\]
Note that $q_\mu$ satisfies the orthogonality condition imposed for elements from $Q_h^{\pm}$, and hence
 $\mbox{span}\{q_\mu\}$ is a subspace in $Q_h^{\pm}$.
Using a trick from~\cite{olshanskii2006analysis}, we decompose  arbitrary $q_h\in Q_h^{\pm}$ into a component collinear with $q_\mu$ and the orthogonal complement in each  phase:
\[
q_h=q_1 + q_0,\quad\text{with}~q_1\in\mbox{span}\{q_\mu\},\quad\text{and}~~ (q_0^{-},1)_{\Omega_{h,i}^{-}}=(q_0^{+},1)_{\Omega_{h,i}^{+}}=0.
\]
Thus, $q_1$ and $ q_0$ are orthogonal with respect to $L^2$ product in the inner domains $\Omega_{h,i}^{\pm}$.
Next, we let $q^\pm=\mu^{-\frac12}_\pm q_0^{\pm}$ in \eqref{aux674} and for  $\bv^{\pm}_h\in  V_h $ given by \eqref{aux674} consider
 $\bv_h^0=\mu_{-}^{\frac12}\bv^{-}_h+\mu_{+}^{\frac12}\bv^{+}_h\in V_h$.
 Then after applying \eqref{Preuss} and summing up, the relations in \eqref{aux674} become
\begin{equation}\label{aux670}
\begin{split}
\mu_{-}^{-1}\| q_0^{-}\|_{{\Omega}_{h}^{-}}^2
  +\mu_{+}^{-1}\| q_0^{+}\|_{{\Omega}_{h}^{+}}^2&\le C\,\left((q_0^{-},\nabla\cdot \bv_h^0)_{{\Omega}^{-}}+(q_0^{+},\nabla\cdot \bv_h^0)_{{\Omega}^{+}}+ b_p(q_0, q_0)\right) ,\\  \|\mu^{\frac12}\nabla\bv_h^0\|_\Omega&\le C_0 \left(\mu_{-}^{-1}\| q_0^{-}\|_{{\Omega}_{h}^{-}}^2
  +\mu_{+}^{-1}\| q_0^{+}\|_{{\Omega}_{h}^{+}}^2\right)^{\frac12},
  \end{split}
\end{equation}
with $C$ from \eqref{Preuss} and $C_0=\max\{C_{-}^{-1},C_{+}^{-1}\}$, both of which are  independent of $\mu_\pm$ and how $\Gamma$ overlaps the background mesh.
In \eqref{aux670},
we also used the fact that supports of $\bv^{-}$ and $\bv^{+}$ do not overlap.
Since ${\rm supp}(\bv_h^\pm)\subset{\Omega}^{\pm}$ and $q_1^\pm$ are constant in ${\Omega}^{\pm}$, integration by parts shows that
\begin{equation}\label{aux705}
  (q_1^\pm,\nabla\cdot \bv_h^0)_{{\Omega}^{\pm}_h} =0.
\end{equation}

Next, we need the following result from Lemma~5.1 in \cite{Kirchhart_Gross2016}:  For all $h\le h_0$ there exists  $\bv_h^1\in V_h$ such that
\begin{equation}\label{aux698}
\begin{split}
\mu_{-}^{-1}\| q_1^-\|_{{\Omega}_{h}^{-}}^2
  +\mu_{+}^{-1}\| q_1^+\|_{{\Omega}_{h}^{+}}^2&= (q_1,\nabla\cdot \bv_h^1)_{{\Omega}^{-}}+(q_1,\nabla\cdot \bv_h^1)_{{\Omega}^{+}},\\ \|\mu^{\frac12}\nabla\bv_h^1\|_\Omega&\le C_1 \left(\mu_{-}^{-1}\| q_1^-\|_{{\Omega}_{h}^{-}}^2
  +\mu_{+}^{-1}\| q_1^+\|_{{\Omega}_{h}^{+}}^2\right)^{\frac12},
  \end{split}
\end{equation}
with $h_0>0$ and $C_1>0$ independent of $\mu_\pm$ and how $\Gamma$ overlaps the background mesh.
The above result follows from the classical inf-sup stability condition for $\mathbf{P}_2-P_1$ Taylor--Hood elements
and a simple scaling and interpolation argument. See \cite{Kirchhart_Gross2016} for details.

As the next step, set $\bv_h=\tau\bv_h^0+ \bv_h^1$ with some $\tau>0$  and proceed with calculations using \eqref{aux705}, \eqref{aux670}, \eqref{aux698}, and the Cauchy-Schwartz inequality:
\[
\begin{split}
(q_h^{-},&\nabla\cdot \bv_h)_{{\Omega}^{-}}+(q_h^{+},\nabla\cdot \bv_h)_{{\Omega}^{+}}\\
&=
(q_1^{-},\nabla\cdot \bv_h^1)_{{\Omega}^{-}}+(q_1^{+},\nabla\cdot \bv_h^1)_{{\Omega}^{+}}+
\tau(q_0^{-},\nabla\cdot \bv_h^0)_{{\Omega}^{-}}+\tau(q_0^{+},\nabla\cdot \bv_h^0)_{{\Omega}^{+}}\\
&\qquad+
(q_0^{-},\nabla\cdot \bv_h^1)_{{\Omega}^{-}}+(q_0^{+},\nabla\cdot \bv_h^1)_{{\Omega}^{+}}
\\
&\ge  \mu_{-}^{-1}\| q_1^{-}\|_{{\Omega}_{h}^{-}}^2
  +\mu_{+}^{-1}\| q_1^{+}\|_{{\Omega}_{h}^{+}}^2
  + \tau C^{-1}\left(\mu_{-}^{-1}\| q_0^{-}\|_{{\Omega}_{h}^{-}}^2
  +\mu_{+}^{-1}\| q_0^{+}\|_{{\Omega}_{h}^{+}}^2\right)  -\tau b_p(q_0, q_0)  \\
  &\qquad
  -\left(\mu_{-}^{-1}\| q_0^{-}\|_{{\Omega}^{-}_h}^2
  +\mu_{+}^{-1}\| q_0^{+}\|_{{\Omega}^{+}_h}^2\right)^{\frac12} d^{\frac12}\|\mu^{\frac12}\nabla\bv_h^1\|_\Omega\\
  & \ge
 \mu_{-}^{-1}\| q_1^{-}\|_{{\Omega}_{h}^{-}}^2
  +\mu_{+}^{-1}\| q_1^{+}\|_{{\Omega}_{h}^{+}}^2
  + \tau C^{-1}\left(\mu_{-}^{-1}\| q_0^{-}\|_{{\Omega}_{h}^{-}}^2
  +\mu_{+}^{-1}\| q_0^{+}\|_{{\Omega}_{h}^{+}}^2\right)-\tau  b_p(q_0, q_0)\\
  &\qquad  -\left(\mu_{-}^{-1}\| q_0^{-}\|_{{\Omega}^{-}_h}^2
  +\mu_{+}^{-1}\| q_0^{+}\|_{{\Omega}^{+}_h}^2\right)^{\frac12}\,C_1d^{\frac12} \left(\mu_{-}^{-1}\| q_1^{-}\|_{{\Omega}_{h}^{-}}^2
  +\mu_{+}^{-1}\| q_1^{+}\|_{{\Omega}_{h}^{+}}^2\right)^{\frac12}\\
  & \ge
 \frac{1}{2} \left(\mu_{-}^{-1}\| q_1^{-}\|_{{\Omega}_{h}^{-}}^2
  +\mu_{+}^{-1}\| q_1^{+}\|_{{\Omega}_{h}^{+}}^2\right)
  + \left(\frac{\tau}{C}-\frac{C_1^2d}{2}\right) \left(\mu_{-}^{-1}\| q_0^{-}\|_{{\Omega}_{h}^{-}}^2
  +\mu_{+}^{-1}\| q_0^{+}\|_{{\Omega}_{h}^{+}}^2\right) -\tau b_p(q_0, q_0).
\end{split}
\]
We set $\tau$ such that  $\frac{\tau}{C}-\frac{C_1^2d}{2}=\frac12$ and note that  $b_p(q_0, q_0)=b_p(q_h, q_h)$.
Using this and the orthogonality condition for $q_0$, we get
\begin{equation}\label{aux730}
\begin{split}
(q_h,&\nabla\cdot \bv_h)_{{\Omega}^{-}}+(q_h,\nabla\cdot \bv_h)_{{\Omega}^{+}}\\
&\ge
 \frac{1}{2}\left(\mu_{-}^{-1}\| q_1^{-}\|_{{\Omega}_{h}^{-}}^2
  +\mu_{+}^{-1}\| q_1^{+}\|_{{\Omega}_{h}^{+}}^2\right)
  + \frac12 \left(\mu_{-}^{-1}\| q_0^{-}\|_{{\Omega}_{h}^{-}}^2
  +\mu_{+}^{-1}\| q_0^{+}\|_{{\Omega}_{h}^{+}}^2\right) -\tau b_p(q_h, q_h)\\
  &\ge
 \frac{1}{2}\left(\mu_{-}^{-1}\| q_1^{-}\|_{{\Omega}_{h,i}^{-}}^2
  +\mu_{+}^{-1}\| q_1^{+}\|_{{\Omega}_{h,i}^{+}}^2\right)
  + \frac12 \left(\mu_{-}^{-1}\| q_0^{-}\|_{{\Omega}_{h}^{-}}^2
  +\mu_{+}^{-1}\| q_0^{+}\|_{{\Omega}_{h,i}^{+}}^2\right) -\tau  b_p(q_h, q_h)\\
 & =
 \frac{1}{2}\left(\mu_{-}^{-1}\| q_h^{-}\|_{{\Omega}_{h,i}^{-}}^2
  +\mu_{+}^{-1}\| q_h^{+}\|_{{\Omega}_{h,i}^{+}}^2\right)
  -\tau  b_p(q_h, q_h)\\
   & \ge
 \frac{1}{2C}\left(\mu_{-}^{-1}\| q_h^{-}\|_{{\Omega}_{h}^{-}}^2
  +\mu_{+}^{-1}\| q_h^{+}\|_{{\Omega}_{h}^{+}}^2\right)
  -\left(\tau+\frac12\right) b_p(q_h, q_h).
 \end{split}
\end{equation}
Using $( q_0^{\pm},q_1^{\pm})_{{\Omega}_{h,i}^{\pm}}=0$, $|\Omega^\pm_h\setminus \Omega^\pm_{h,i}|\le c\,h$ and so $\|q_1^{\pm}\|_{\Omega^\pm_h\setminus \Omega^\pm_{h,i}}\le c h^{\frac12}\|q_1^{\pm}\|_{\Omega^\pm_h}$, we estimate
\begin{multline}\label{aux780}
|\mu_{-}^{-1}( q_0^{-},q_1^{-})_{{\Omega}_{h}^{-}}  +\mu_{+}^{-1}( q_0^{+},q_1^{+})_{{\Omega}_{h}^{+}}| \\ \le
c\,h^{\frac12}\left(\mu_{-}^{-1}\| q_0^{-}\|_{{\Omega}_{h}^{-}}^2
  +\mu_{+}^{-1}\| q_0^{+}\|_{{\Omega}_{h}^{+}}^2\right)^{\frac12}\left(\mu_{-}^{-1}\| q_1^{-}\|_{{\Omega}_{h}^{-}}^2
  +\mu_{+}^{-1}\| q_1^{+}\|_{{\Omega}_{h}^{+}}^2\right)^{\frac12}.
\end{multline}
From \eqref{aux670}, \eqref{aux698}, and \eqref{aux780}, we also get the following upper bound  for $\bv_h$,
\begin{equation}\label{aux762}
\begin{split}
\|\mu^{\frac12}\nabla\bv_h\|_{\Omega}^2&\le 2(\|\mu^{\frac12}\tau\nabla\bv_h^0\|_{\Omega}^2+ \|\mu^{\frac12}\nabla\bv_h^1\|_{\Omega}^2)\\
&\le
2\tau^2 C_0^2\left(\mu_{-}^{-1}\| q_0^{-}\|_{{\Omega}_{h}^{-}}^2
  +\mu_{+}^{-1}\| q_0^{+}\|_{{\Omega}_{h}^{+}}^2\right)+
  2C_1^2\left(\mu_{-}^{-1}\| q_1^{-}\|_{{\Omega}_{h}^{-}}^2
  +\mu_{+}^{-1}\| q_1^{+}\|_{{\Omega}_{h}^{+}}^2\right)\\
 &\le
  \frac{2\max\{\tau^2 C_0^2, C_1^2\}}{1-c\,h^{\frac12}} \left(\mu_{-}^{-1}\| q_h^{-}\|_{{\Omega}_{h}^{-}}^2
  +\mu_{+}^{-1}\| q_h^{+}\|_{{\Omega}_{h}^{+}}^2\right).
\end{split}
\end{equation}
The assertion of the lemma follows from \eqref{aux730} and \eqref{aux762} after simple calculations. \\
\end{proof}

The next lemma  shows the uniform coercivity of the symmetric form $a_h(\bu_h,\bv_h)$
in \eqref{eq:sa_h} on $ V_h^\pm\times V_h^\pm$.

\begin{lemma}\label{Lem3} If  $\gamma=O(1)$ in \eqref{eq:a_n} is sufficiently large, then it holds
\begin{equation}\label{coerc}
a_h(\bu_h,\bu_h)\ge C\,\left(\mu_{-}\| \bD(\bu_h^-)\|_{{\Omega}_{h}^{-}}^2
  +\mu_{+}\|\bD(\bu_h^{+})\|_{{\Omega}_{h}^{+}}^2+ h^{-1}\|\{\mu\}[\bu_h\cdot\bn]\|_\Gamma^2 +  f\|[\bP\bu_h]\|_\Gamma^2\right)
\end{equation}
$\forall~\bu_h\in V_h^\pm$, with $C>0$  independent of $\mu_\pm$, $h$, {f,} and the position of $\Gamma$ with respect to
the background mesh.
\end{lemma}
\begin{proof}
For the proof, we need the local trace inequality in $T\in\T^\Gamma_h$~(see, e.g. \cite{guzman2018inf,Hansbo02}):
\begin{equation}\label{trace}
\|v\|_{T\cap\Gamma}\le C(h_T^{-\frac12}\|v\|_{T}+h_T^{\frac12}\|\nabla v\|_{T}),\quad ~~\forall~v\in H^1(T),
\end{equation}
with a constant $C$ independent of $v$, $T$,  how $\Gamma$ intersects $T$, and $h_T<h_0$ for some arbitrary but fixed $h_0$.
We also need the following estimate
\begin{equation}\label{aux7}
\|\bD( \bv_h^\pm)\|^2_{L^2(\Omega_h^\pm)}\le C(\|\bD(\bv_h^\pm)\|^2_{L^2(\Omega^\pm)}+ \bJ^{\pm}(\bv_h^\pm,\bv_h^\pm)\,),
\end{equation}
which follows from \eqref{Preuss} by applying it componentwise and further using FE inverse inequality (note $h^{-2}$ scaling in the definition of $\bJ^{\pm}$ in \eqref{eq:a_j}).
Applying \eqref{trace}, finite element inverse inequalities and \eqref{aux7}, we can bound the interface term
\begin{equation*}
 \begin{split}
 \langle \{ \mu \bn^T \bD(\bv_h) \bn \} ,  &[\bu_h \cdot \bn] \rangle_\Gamma = \langle \mu_{-} \bn^T \bD(\bv_h^{-}) \bn  ,  [\bu_h \cdot \bn] \rangle_\Gamma\\
 &\le
 \sum_{T \in \mathcal{T}_h^\Gamma}\left[\frac{h_T\delta}{2}\|\mu_{-}^{\frac12} \bn^T \bD(\bv_h^{-}) \bn \|^2_{T\cap\Gamma}
 +\frac 1{2h_T\delta}\|\mu_{-}^{\frac12} [\bu_h \cdot \bn] \|^2_{T\cap\Gamma}
 \right]\\
&\le
\frac{\delta}{2}\|\mu_{-}^{\frac12} \bn^T \bD(\bv_h^{-}) \bn \|^2_{\Omega_h^{-}}+
\frac{1}{h_T\delta} \{ \mu \} |[\bu_h \cdot \bn]|^2_\Gamma,\quad \forall~\delta>0,~\bu_h,\bv_h\in V_h^\pm.
\end{split}
\end{equation*}
This estimate with $\bv_h = \bu_h$ and with $\delta>0$ sufficiently small, together with the definition of the bilinear form $a_h(\bu_h,\bu_h)$, allows to show its coercivity.
\end{proof}

We further need the continuity result for the velocity stabilization form contained in the next lemma.
\begin{lemma}\label{Lem2} It holds
\[
a_p(\bv_h,\bv_h)\le C\,\left(\mu_{-}\|\bD(\bv^-_{h})\|_{{\Omega}_{h}^{-}}^2
  +\mu_{+}\|\bD(\bv^{+}_{h})\|_{{\Omega}_{h}^{+}}^2\right)\qquad\forall~\bv_h\in V^\pm_h,
\]
with $C>0$  independent of $\mu_\pm$, $h$, and the position of $\Gamma$ in the background mesh.
\end{lemma}
\begin{proof} For any $\bv=\bv_h^{-}\in V^{-}_h$, facet $e\in \mathcal{E}_h^{\Gamma, -}$ and
the corresponding patch $\omega_e$ formed by two tetrahedra $T_1$ and $T_2$, it holds
\begin{equation*}
\|\bv_1^e - \bv_2^e\|^2_{\omega_e}=\|\bv_1 - \bv_2^e\|^2_{T_1}+\|\bv_1^e - \bv_2\|^2_{T_2}\le (1+c)\|\bv_1 - \bv_2^e\|^2_{T_1},
\end{equation*}
where the constant $c$ depends only on shape regularity of the tetrahedra, since $\bv_1^e - \bv_2$ on $T_2$ is the canonical  polynomial extension of $\bv_1 - \bv_2^e$ from $T_1$.

Now, we need the following local Korn's inequality:
\begin{equation}\label{lKorn}
\|\nabla \bv\|_T\le C  \|\bD(\bv)\|_T,\quad \forall~\bv\in H^1(T)^d,~~\text{s.t.}~\bv=0~\text{on any face of}~ T\in\T_h,
\end{equation}
where $C$ depends only on shape regularity of $T$. The result in \eqref{lKorn} follows from eq.~(3.3) in~\cite{brenner2004korn} and  the observation that vector fields  vanishing on any face $T$ support only zero rigid motions. A simple scaling argument also proves the local Poincare inequality:
\begin{equation}\label{lPoin}
\|\bv\|_T\le C  h^2_T\|\nabla\bv\|_T,\quad \forall~\bv\in H^1(T)^d,~~\text{s.t.}~\bv=0~\text{on any face of}~ T\in\T_h,
\end{equation}
where $C$ depends only on shape regularity of $T$.
Applying  \eqref{lKorn}, \eqref{lPoin} and triangle  inequalities on $T_1$ for $\bv_1 - \bv_2^e$ which vanishes on $e$ (a face of $T_1$), we obtain:
\begin{align}
\|\bv_1 - \bv_2^e\|^2_{T_1}&\le C_p h^2 \|\bD(\bv_1 - \bv_2^e)\|^2_{T_1} \le 2 C_p h^2 (\|\bD\bv_1\|^2_{T_1} + \|\bD\bv_2^e\|^2_{T_1}) \cl
&\le 2 C_p h^2 (\|\bD\bv_1\|^2_{T_1} + c\,\|\bD\bv_2\|^2_{T_2}),
\end{align}
where for the last inequality we again use shape regularity and the fact that $\bD\bv_2^e= (\bD\bv_2)^e$.
Thus, we see that $\|\bv_1^e - \bv_2^e\|^2_{\omega_e}\le c\,h^2 \|\bD\bv\|^2_{\omega_e} $, with some $c$ depending only on shape regularity.
Summing up over all $e\in \mathcal{E}_h^{\Gamma, -}$  leads to the required upper bound for $\bJ_h^{-}(\bv,\bv)$: $\bJ_h^{-}(\bv,\bv)\le C \|\bD(\bv)\|_{\Omega_h^{-}}$. Repeating the same argument
for the edges in $\mathcal{E}_h^{\Gamma, +}$ and summing up the two bounds scaled by viscosity coefficients  proves  the lemma.\\
\end{proof}

The finite element  problem  \eqref{FE_formulation} can be equivalently  formulated as follows: Find $\{\bu_h, p_h\} \in V^\pm_h \times Q^\pm_h$ such that
\begin{equation}\label{eqFE}
\A(\bu_h,p_h;\bv_h,q_h)=r_h(\bv_h),\quad\forall~\{\bv_h, q_h\} \in V^\pm_h \times Q^\pm_h
\end{equation}
with
\begin{equation*}
\A(\bu_h,p_h;\bv_h,q_h)=a_h(\bu_h,\bv_h)+b_h(\bv_h,p_h)- b_h(\bu_h,q_h) + b_p(p_h, q_h).
\end{equation*}

Lemmas~\ref{Lem1}--\ref{Lem2} enable us to show the inf-sup stability of the bilinear form $\A$.
The stability result is formulated using the following composite norm:
\[
\|\bv,q\|^2:= \mu_{-}\| \bD(\bv^-)\|_{{\Omega}_{h}^{-}}^2
  +\mu_{+}\|\bD(\bv^{+})\|_{{\Omega}_{h}^{+}}^2+ h^{-1}\|\{\mu\}[\bv\cdot\bn]\|_\Gamma^2 +  f\|[\bP\bv]\|_\Gamma^2 + \mu_{-}^{-1}\| q^{-}\|_{{\Omega}_{h}^{-}}^2
  +\mu_{+}^{-1}\| q^{+}\|_{{\Omega}_{h}^{+}}^2
\]
for $\bv\in V^\pm_h,~q\in Q^\pm_h$.

\begin{theorem}\label{Th1}There exists $h_0>0$ such that for all $h<h_0$ it holds
\begin{equation*}
  \sup_{\{\bv_h,q_h\}\in V^\pm_h\times Q^\pm_h}\frac{\A(\bu_h,p_h;\bv_h,q_h)}{\|\bv_h,q_h\|}
  \ge C\,\|\bu_h,p_p\|,\quad\forall\,\{\bu_h,p_h\}\in V^\pm_h\times Q^\pm_h,
\end{equation*}
with $h_0>0$ and $C>0$  independent of $\mu_\pm$, $h$, {f,} and the position of $\Gamma$ in the background mesh.
\end{theorem}
\begin{proof} For a given $p_h\in Q^\pm_h$, Lemma~\ref{Lem1} implies the existence of such $\bw_h\in V_h$ that
\begin{equation}\label{aux797}
 b_h(\bw_h,p_h)+ b_p(p_h, q_h)\ge  c\,\left(\mu_{-}^{-1}\| p_h^{-}\|_{{\Omega}_{h}^{-}}^2
  +\mu_{+}^{-1}\| p_h^{+}\|_{{\Omega}_{h}^{+}}^2\right)
\end{equation}
and
\begin{equation}\label{aux802}
\|\mu^{\frac12}\nabla\bw_h\|_{\Omega}^2\le C\,\left(\mu_{-}^{-1}\| p_h^{-}\|_{{\Omega}_{h}^{-}}^2
  +\mu_{+}^{-1}\| p_h^{+}\|_{{\Omega}_{h}^{+}}^2\right),
\end{equation}
with some positive $c$, $C$  independent of $\mu$ and how $\Gamma$ overlaps the background mesh.
Next, we extend the finite element function $\bw_h\in V_h$ to the element of the product space $\widehat\bw_h\in V_h^\pm$ by setting $\widehat\bw_h^\pm=\bw_h|_{\Omega_{h}^\pm}\in V_h^\pm$.
We let $\bv_h=\bu_h+\tau\widehat\bw_h$ for some $\tau>0$ and $q_h=p_h$.
Using the definition of the form $\A$ and \eqref{aux797}, we calculate
\begin{equation}\label{aux808}
\begin{split}
   \A(\bu_h&,p_h;\bv_h,q_h)  = a_h(\bu_h,\bu_h)+ \tau a_h(\bu_h,\widehat\bw_h) + \tau b_h(\widehat\bw_h,p_h) + b_p(p_h, p_h) \\
     & \ge \frac{1}{2}a_h(\bu_h,\bu_h)-\frac{\tau^2}{2} a_h(\widehat\bw_h,\widehat\bw_h) + \min\{\tau,1\}\,c\, \left(\mu_{-}^{-1}\| p_h^{-}\|_{{\Omega}_{h}^{-}}^2
  +\mu_{+}^{-1}\| p_h^{+}\|_{{\Omega}_{h}^{+}}^2\right),
\end{split}
\end{equation}
where we used the Cauchy-Schwartz inequality:
\[
\tau a_h(\bu_h,\widehat\bw_h)\le \tau |a_h(\bu_h,\bu_h)|^{\frac12} |a_h(\widehat\bw_h,\widehat\bw_h)|^{\frac12}
\le   \frac12a_h(\bu_h,\bu_h) + \frac{\tau^2}{2} a_h(\widehat\bw_h,\widehat\bw_h).
\]
Note that it holds $[\widehat\bw_h\cdot\bn]=0$ and $[\bP\widehat\bw_h]=0$ on $\Gamma$.
Since all Nitsche and `friction' terms in $a_h(\widehat\bw_h,\widehat\bw_h)$ vanish,  the results of the Lemma~\ref{Lem2} and estimate \eqref{aux802} imply  the upper bound
\[
a_h(\widehat\bw_h,\widehat\bw_h)\le C\, \|\mu^{\frac12}\nabla\widehat\bw_h\|_{\Omega}^2\le C\,\left(\mu_{-}^{-1}\| p_h^{-}\|_{{\Omega}_{h}^{-}}^2
  +\mu_{+}^{-1}\| p_h^{+}\|_{{\Omega}_{h}^{+}}^2\right).
\]
Using it in \eqref{aux808} and choosing $\tau>0$ small enough, but independent of all problem parameters, leads us to the lower bound
\begin{equation}\label{aux847}
   \A(\bu_h,p_h;\bv_h,q_h)  \ge \frac{1}{2}a_h(\bu_h,\bu_h)+ c\, \left(\mu_{-}^{-1}\| p_h^{-}\|_{{\Omega}_{h}^{-}}^2
  +\mu_{+}^{-1}\| p_h^{+}\|_{{\Omega}_{h}^{+}}^2\right)\ge c\,\|\bu_h,p_h\|^2,
\end{equation}
with some $c>0$  independent of $\mu_\pm$, $h$, and the position of $\Gamma$ in the background mesh. For the last inequality, we used \eqref{coerc}.

Finally, by the construction of $\bv_h$ and  thanks to \eqref{aux802} it is straightforward to see the upper bound:
\[
\|\bv_h,q_h\|\le c\,\|\bu_h,p_h\|.
\]
This combined  with  \eqref{aux847} proves the theorem.\\
\end{proof}

The stability of the finite element solution in the composite norm immediately follows from \eqref{eqFE} and Theorem~\ref{Th1}:
\[
\|\bu_h,p_h\|\le C\, \sup_{\{\bv_h,q_h\}\in V^\pm_h\times Q^\pm_h}\frac{|r_h(\bv_h)|}{\|\bv_h,q_h\|},
\]
where on the right-hand side we see the dual norm of the functional $r_h$ and constant $C$, which is independent
of the mesh size $h$, the ratio of the viscosity coefficients  $\mu_\pm$, and the position of $\Gamma$ in the background mesh.

\section{Error analysis}\label{sec:error}

The stability result shown in Sec.~\ref{sec:stab} and interpolation properties of finite elements enable us to prove optimal order convergence with uniformly bounded constants.

We assume in this section  that the solution to problem  \eqref{eq:Stokes1}--\eqref{eq:scc4} is piecewise smooth in the following sense: $\bu^\pm \in H^{k+2}(\Omega^\pm)^d$ and $p^\pm \in H^{k+1}(\Omega^\pm)$. For the sake of notation,
we define the following semi-norm
\begin{equation}\label{seminorm}
\|\bu,p\|_\ast=\left(\mu_{-}|\bu^{-}|_{H^{k+2}(\Omega^{-})}^2+\mu_{+}|\bu^{+}|_{H^{k+2}(\Omega^{+})}^2+ \mu_{-}^{-1}|p^{-}|_{H^{k+1}(\Omega^{-})}^2+\mu_{+}^{-1}|p^{+}|_{H^{k+1}(\Omega^{+})}^2\right)^{\frac12}.
\end{equation}
 Since we assume $\Gamma$ to be at least Lipschitz, there exist extensions  $\cE  \bu^\pm$ and $\cE  p^\pm$ of the solution from each phase to $\mathbb{R}^d$ such that  $\cE  \bu^\pm \in  H^{k+2}(\mathbb{R}^d)^3$, $\cE  p^\pm \in H^{k+1}(\mathbb{R}^d)$.
The corresponding norms are bounded  as follows
\begin{equation}\label{eq:ext_bounds}
\|\cE  \bu^\pm\|_{ H^{k+2}(\mathbb{R}^d)}\le C\,\|\bu^\pm\|_{ H^{k+2}(\Omega^\pm)}, \quad \|\cE  p^\pm\|_{ H^{k+1}(\mathbb{R}^d)}\le C\,\| p^\pm\|_{ H^{k+1}(\Omega^\pm)}
\end{equation}
See \cite{SteinBook}.
Denote by $I_h \bu^\pm$ the Scott-Zhang interpolants of $\cE \bu^\pm$ onto $ V^\pm_h$ and $I_h \bu:=\{I_h \bu^{-},I_h \bu^{+}\}$.
Same notation $I_h  p^\pm$ will be used for the  Scott-Zhang interpolants of $\cE p^\pm$ onto $Q^\pm_h$. For the pressure interpolants,
we can always satisfy the orthogonality condition of $Q^\pm_h$ by choosing
a suitable additive constant in the definition of $p$.

Applying trace inequality  \eqref{trace}, standard approximation properties of $I_h$, and bounds \eqref{eq:ext_bounds}, one obtains the approximation property in the product norm:
\begin{equation}\label{approxA}
\|\bu-I_h\bu,p-I_hp\| \le C\,h^{k+1} \|\bu,p\|_\ast.
\end{equation}
The following continuity result is the immediate consequence of the Cauchy--Schwatz inequality:
\begin{multline}\label{contA2}
\A(\bu-I_h \bu,p-I_h p;\,\bv_h,q_h) \le C\,\|\bu-I_h\bu,p-I_h p\|\|\bv_h,q_h\|\\ +| \langle \{ \mu \bn^T \bD(\bv_h) \bn \} ,  [(\bu-I_h \bu) \cdot \bn] \rangle_\Gamma
+ \langle \{ \mu \bn^T \bD(\bu-I_h \bu) \bn \} ,  [\bv_h \cdot \bn] \rangle_\Gamma|,
\end{multline}
for all $\{\bv_h,q_h\}\in  V_h^\pm\times Q_h^\pm$.
The last term on the right-hand side in \eqref{contA2} needs a special treatment.  Applying the Cauchy--Schwatz,
inequalities \eqref{trace} and \eqref{aux7}, FE inverse inequalities and approximation properties of the interpolants, we get
\begin{equation}\label{contA3}
\begin{split}
|\langle \{ \mu \bn^T \bD(\bv_h) \bn \} ,  [(\bu-I_h \bu) \cdot \bn] \rangle_\Gamma|\le C\,h^{k+1}\|\bu,0\|_\ast\|\bv_h,0\|,\\
|\langle \{ \mu \bn^T \bD(\bu-I_h \bu) \bn \} ,  [\bv_h \cdot \bn] \rangle_\Gamma|\le C\,h^{k+1}\|\bu,0\|_\ast\|\bv_h,0\|.
\end{split}
\end{equation}
The consistency of the stabilization term is formalized in the estimates that follow from~\cite[lemma 5.5]{lehrenfeld2018eulerian}:
For $p^- \in H^{k+1}(\Omega^{-})$,  $\bu^- \in H^{k+2}(\Omega^{-})^d$, it holds
  \begin{equation} \label{eq:shnw}
    J_h^{-}(p^-,p^-) \le C\, h^{2k+2} \|p^-\|_{H^{k+1}(\Omega^{-})}^2,\quad  \bJ_h^{-}(\bu^-,\bu^-) \le C\, h^{2k+2} \|\bu^-\|_{H^{k+2}(\Omega^{-})}^2.
  \end{equation}
The above estimates and the stability of the interpolants also imply
 \begin{equation} \label{eq:shnIw}
 \begin{split}
    J_h^{-}(p^--I_hp^-,p^--I_hp^-) &\le C\, h^{2k+2} {|p^-|}_{H^{k+1}(\Omega^{-})}^2,\\  \bJ_h^{-}(\bu^--I_h\bu^-,\bu^--I_h\bu^-) &\le C\, h^{2k+2} {|\bu^-|}_{H^{k+2}(\Omega^{-})}^2.
    \end{split}
 \end{equation}
Similar estimates to \eqref{eq:shnw}, \eqref{eq:shnIw} hold  for $J_h^{+}$ and $\bJ_h^{+}$ with $p^+ \in H^{k+1}(\Omega^{+})$,  $\bu^+  \in H^{k+2}(\Omega^{+})^d$,
which can be combined with suitable weights to yield
  \begin{equation} \label{eq:apbp}
    b_p(p-I_hp,p-I_hp)+a_p(\bu-I_h\bu,\bu-I_h\bu)  \le C\,h^{2k+2} \|\bu,p\|_\ast^2.
  \end{equation}

Denote the  error functions by $ \be_u=\cE  \bu-\bu_h$ and $e_p=\cE  p-p_h$.
Galerkin orthogonality holds up to the consistency terms
\begin{equation} \label{Galerkin}
\A(\be_u,e_p;\,\bv_h,q_h)= b_p(p-I_hp,q_h)+a_p(\bu-I_h\bu,\bv_h),
\end{equation}
for all $\bv_h \in V_h^\pm$ and $ q_h \in Q_h^\pm$.

The result of Lemma~\ref{Lem3}, \eqref{eq:apbp} and the trivial bound $b_p(q_h,q_h)\le C\|\mathbf{0},q_h\|^2$ imply
the following estimate for the consistency term on the right-hand side of \eqref{Galerkin}:
\begin{equation} \label{consist}
\begin{split}
|b_p(p&-I_hp,q_h)+a_p(\bu-I_h\bu,\bv_h)|\\ &
\le |b_p(p-I_hp,p-I_hp)|^{\frac12}|b_p(q_h,q_h)|^{\frac12}+|a_p(\bu-I_h\bu,\bu-I_h\bu)|^{\frac12}|a_p(\bv_h,\bv_h)|^{\frac12} \\
&\le
 C\,h^{k+1}\|\bu,p\|_\ast\|\bv_h,q_h\|,
\end{split}
\end{equation}

The optimal order error estimate in the energy norm is given in the next theorem.

\begin{theorem}\label{Th2} For sufficiently  regular $\bu,p$ solving problem \eqref{eq:Stokes1}--\eqref{eq:scc4} and $\bu_h,p_h$
solving problem \eqref{FE_formulation}, the following error estimate holds:
\begin{equation}\label{error_est}
\|\bu-\bu_h,p-p_h\|\le \,C h^{k+1}\|\bu,p\|_\ast,
\end{equation}
with a constant $C$ independent of $h$, the values of viscosities $\mu_\pm$, slip coefficient $f\ge0$, and the position of $\Gamma$ with respect to the triangulation $\T_h$.
\end{theorem}
\begin{proof}
This result follows  by standard arguments (see, for example, section 2.3 in \cite{ern2013theory}) from the inf-sup stability results of Theorem~\ref{Th1}, continuity estimates  \eqref{contA2} and \eqref{contA3}, Galerkin orthogonality and consistency \eqref{Galerkin}--\eqref{consist}, and approximation properties \eqref{approxA}.
\end{proof}

\begin{remark}\label{rem1}\rm 
If we consider  using isoparametric elements to handle  numerical integration over cut cells (see section~\ref{s:numint}),
then the Sobolev seminorms in the definition of $\|\bu,p\|_\ast$  on the  right-hand side 
in \eqref{error_est} should be replaced by the full Sobolev norms of the same order; 
see the error analysis of the isoparametric unfitted FEM in \cite{lehrenfeld2018analysis}.
\end{remark}

\section{Numerical results}\label{num_res}

The aim of the numerical results collected in this section is twofold:
(i) support the theoretical results presented in Sec.~\ref{sec:error}
and (ii) provide evidence of the robustness of the proposed finite element approach
with respect to the contrast in viscosity, slip coefficient value, and position of the interface
relative to the fixed computational mesh.

For the averages in \eqref{curly_av}-\eqref{angle_av}, we set $\alpha=0$ and $\beta=1$ for all the numerical experiments
since we have $\mu_{-}\le\mu_+$.
Recall that this is the choice for the analysis carried out in Sec.~\ref{sec:stab}
and \ref{sec:error}.
In addition, we set $\gamma^\pm_\bu = 0.05$, $\gamma_p^\pm = 0.05$, and $\gamma=40$.
The value of all other parameters will depend on the specific test.

For all the results presented below, we will report the $L^2$ error and a weighted $H^1$ error for the velocity defined as
\begin{equation}\label{eq:weighted_H1}
\left(2\mu_{-}\|D(\bu-\bu_h^-)\|_{L^2(\Omega^{-})}^2+2\mu_{+}\|D(\bu-\bu_h^+)\|_{L^2(\Omega^{+})}^2\right)^{\frac12},
\end{equation}
and a weighted $L^2$ error for the pressure defined as
\begin{equation}\label{eq:weighted_L2}
\left(
\mu^{-1}_{-}\|p-p_h^-\|_{L^2(\Omega^{-})}^2 + \mu^{-1}_{+}\|p-p_h^+\|_{L^2(\Omega^{+})}^2
\right)^{\frac12}.
\end{equation}

\subsection{2D tests}

First, we perform a series of tests in 2D. For all the tests, the domain $\Omega$ is square $[-1,1]\times [-1,1]$
and interface $\Gamma$ is a circle of radius $2/3$ centered at $\bc=(c_1,c_2)$. Let  
$(x,y)=(\tilde x-c_1,\tilde y-c_2)$, $(\tilde x,\tilde y)\in\Omega$.
The exact solution we consider
is given by:
\begin{align}
p^-&=(x-c_1)^3,  \hspace*{2.4cm} p^+=(x-c_1)^3-\frac{1}{2}, \label{ex_sol2D_p} \\
\bu^-&=g^-(x,y)
\left[
\begin{array}{c}
-y   \\
x
\end{array}
\right],
\qquad \bu^+=g^+(x,y)
\left[
\begin{array}{c}
-y   \\
x
\end{array}
\right],  \label{ex_sol2D_u}
\end{align}
where
\begin{align*}
g^+(x,y)=\frac{3}{4\mu_+}(x^2+y^2), \quad g^-(x,y)=\frac{3}{4\mu_-}(x^2+y^2)+\frac{\mu_--\mu_+}{3\mu_+\mu_-}+\frac{1}{f}.
\end{align*}
The forcing terms $\bbf^-$ and $\bbf^+$ are found by plugging the above solution in \eqref{eq:Stokes1}.
The surface tension coefficient $\sigma$ is set to -0.5.
The value of the other physical parameters will be specified for each test.

We impose a Dirichlet condition \eqref{eq:bcD} on the entire boundary, where function $\bg$
is found from $\bu^+$ in \eqref{ex_sol2D_u}.

\vskip .2cm
\noindent {\bf Spatial convergence.}
First, we check the spatial accuracy of the finite element method described in Sec.~\ref{sec:FE}.
The aim is to validate our implementation of the method and support the theoretical findings
in Sec.~\ref{sec:error}. For this purpose, we consider
exact solution \eqref{ex_sol2D_p}-\eqref{ex_sol2D_u} with $\bc=\boldsymbol{0}$ (i.e., interface $\Gamma$
is a circle centered at the origin of the axes), viscosities $\mu_-=1$ and $\mu_+=10$, and
$f=10$.

We consider structured meshes of quads with six levels of refinement. The initial triangulation has a mesh size
$h = 1/2$ and all the other meshes are obtained by halving $h$ till $h = 1/128$.
We choose to use finite element pairs $\bQ_{2} - Q_1$.
Fig.~\ref{fig:2Dsol} shows the velocity vectors colored with
the velocity magnitude and the pressure computed with
mesh $h = 1/128$.
Fig.~\ref{fig:spatial} shows the $L^2$ error and weighted $H^1$ error
\eqref{eq:weighted_H1} for the velocity and weighted $L^2$ error  \eqref{eq:weighted_L2}
for the pressure against the mesh size $h$.
For the range of mesh sizes under consideration, we observe close to cubic convergence in the $L^2$ norm for the velocity and
quadratic convergence in the weighted $L^2$ norm for the pressure and in the weighted $H^1$ norm for the velocity.

\begin{figure}[hbt!]
\centering
\includegraphics[width=.35\textwidth]{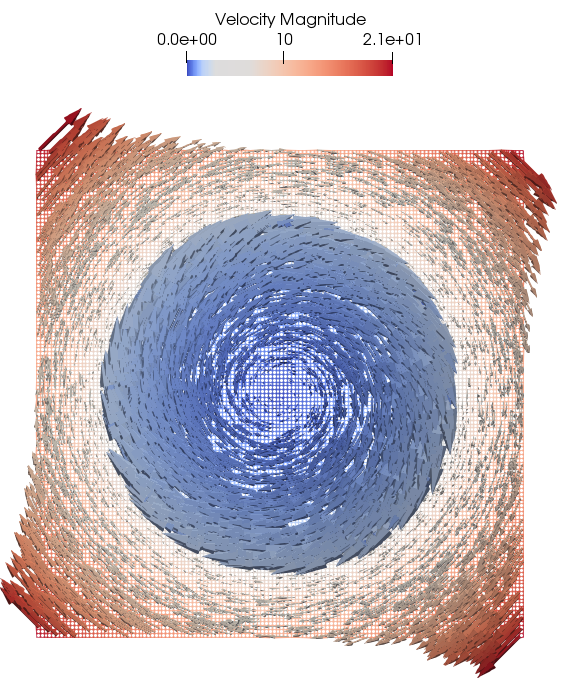}
\includegraphics[width=.34\textwidth]{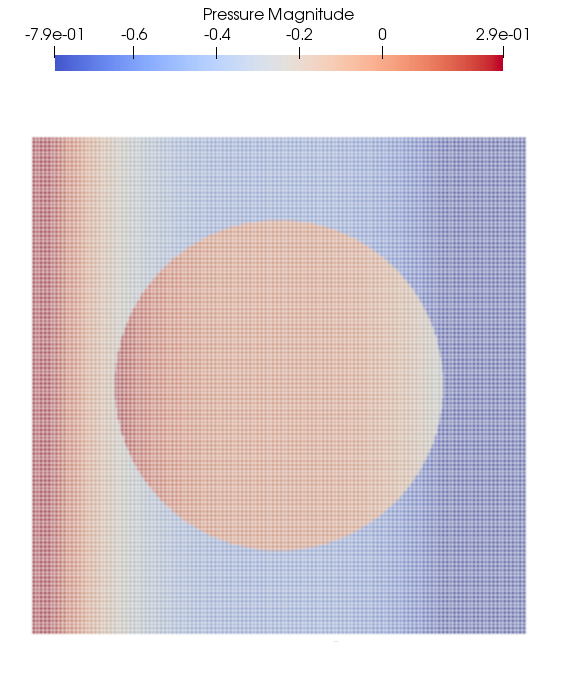}
  \caption{Approximation of exact solution \eqref{ex_sol2D_p}-\eqref{ex_sol2D_u} for $\bc=\boldsymbol{0}$,
  $\mu_-=1$, $\mu_+=10$, and $f=10$, computed with mesh $h = 1/128$:
  velocity vectors colored with the velocity magnitude (left) and pressure (right).}
  \label{fig:2Dsol}
\end{figure}


\begin{figure}[hbt!]
\centering
\includegraphics[width=.55\textwidth]{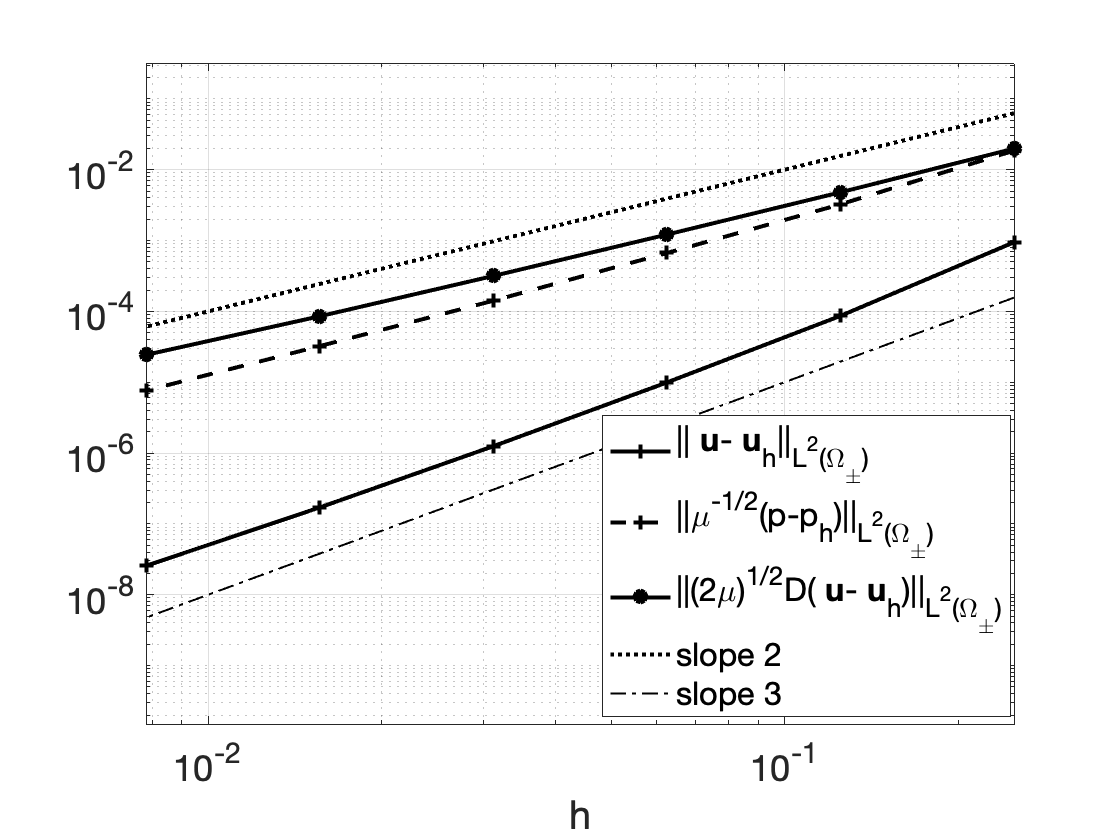}
  \caption{2D test with $\bc=\boldsymbol{0}$, $\mu_-=1$, $\mu_+=10$, and $f=10$: $L^2$ error and weighted $H^1$
  error \eqref{eq:weighted_H1} for the velocity and weighted $L^2$ error \eqref{eq:weighted_L2} for the pressure against the mesh size $h$.}
  \label{fig:spatial}
\end{figure}

{\bf Robustness with respect to the viscosity contrast.} As mentioned in Sec.~\ref{sec:intro},
the case of high contrast for the viscosities in a two-phase problem is especially challenging
from the numerical point of view. To test the robustness of our approach,
we consider exact solution \eqref{ex_sol2D_p}-\eqref{ex_sol2D_u} and
fix $\mu_-=1$, while we let $\mu_+$ vary from 1 to $10^8$.
We set $\bc=\boldsymbol{0}$ and $f=10$.

We consider one of the meshes adopted for
the previous sets of simulations (with  $h=1/64$) and use again
$\bQ_{2} - Q_1$ finite elements.
Fig.~\ref{fig:contrast} (left) shows the $L^2$ error and weighted $H^1$ error \eqref{eq:weighted_H1}
for the velocity and weighted $L^2$ error \eqref{eq:weighted_L2} for the pressure against
the value of $\mu_+$. We observe that all the errors quickly reach a plateau as the
$\mu_+/\mu_-$ ratio increases, after initially decreasing.
These results show that our approach is substantially robust with respect to the viscosity
contrast $\mu_+/\mu_-$.


\begin{figure}[hbt!]
\centering
\includegraphics[width=.49\textwidth]{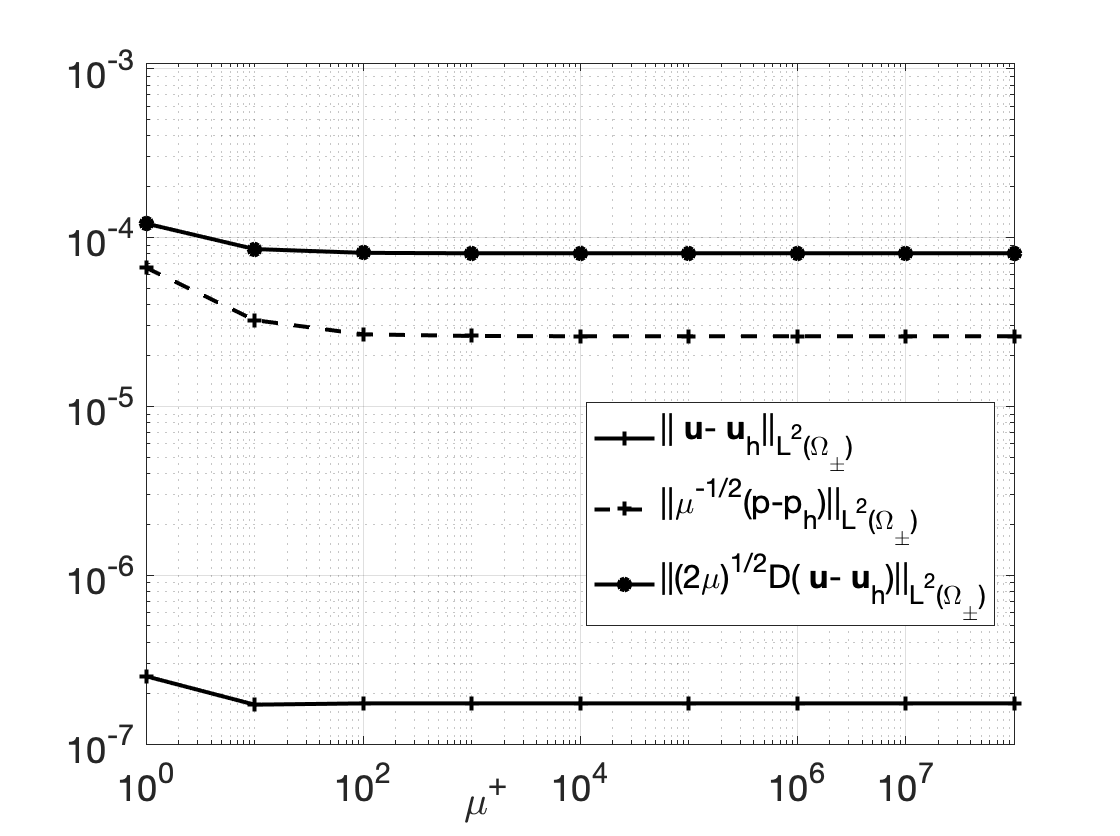}
\includegraphics[width=.49\textwidth]{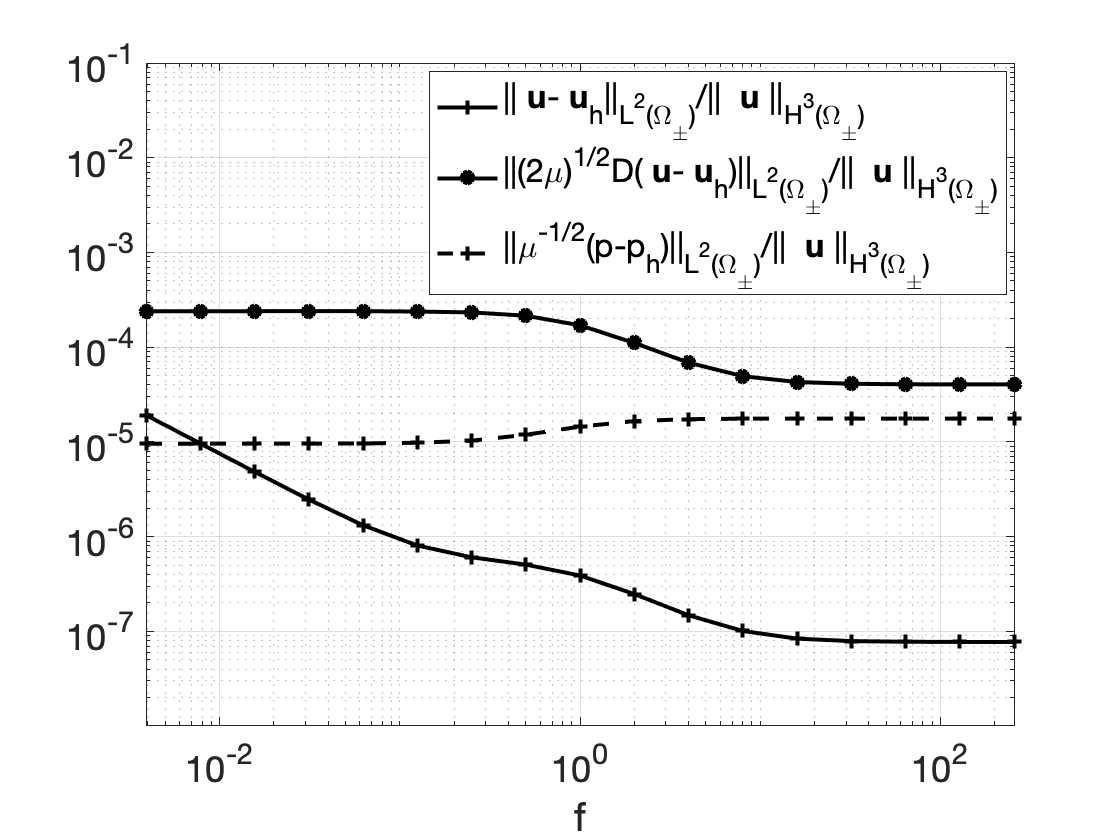}
  \caption{2D test with $\bc=\boldsymbol{0}$ and $\mu_-=1$: $L^2$ error and weighted $H^1$ error \eqref{eq:weighted_H1} for the velocity and weighted $L^2$ error
  \eqref{eq:weighted_L2} for the pressure against the value of $\mu_+$ (left) and against the value of the slip coefficient $f$
  (right).}
  \label{fig:contrast}
\end{figure}

{\bf Robustness with respect to the slip coefficient.} For the next set of simulations, we  consider exact solution \eqref{ex_sol2D_p}-\eqref{ex_sol2D_u}
and let the slip coefficient $f$ in \eqref{eq:scc2}-\eqref{eq:scc3} vary from 1/256 to 256.
Notice that the larger $f$ becomes, the closer the two-phase problem gets to the homogeneous model.
The other parameters are set as follows: $\bc=\boldsymbol{0}$, $\mu_-=1$, and $\mu_+=10$.

We consider again the structured mesh with mesh size $h=1/64$ and
$\bQ_{2} - Q_1$ finite elements.
Fig.~\ref{fig:contrast} (right) shows the $L^2$ error and weighted $H^1$ error \eqref{eq:weighted_H1}
for the velocity scaled by the $H^{3}$ norm of $\bu$ and weighted $L^2$ error \eqref{eq:weighted_L2} for the pressure against
the value of $f$. We observe that the scaled weighted $H^1$ error for the velocity does not vary 
substantially as $f$ varies, while the other two errors increase as $f$ decreases. 
When $f$ goes to zero, the external phase loses its control over tangential motions 
in the internal fluid on $\Gamma$, thus allowing for 
purely rigid rotations in the perfectly circular $\Omega^{-}$; see the definition of $\bu^{-}$ in 
\eqref{ex_sol2D_u}. 
While the seminorm $\|\bu,p\|_\ast$ appearing  on the right-hand side in \eqref{error_est} 
remains the same, the full Sobolev norm $\|\bu^{-}\|_{k+2}$ grows as $O(f^{-1})$. Since we use isoparametric unfitted FE, 
we indeed see the uniform error bound with respect to $f\to 0$ if we normalize the error by the full Sobolev norm of the solution. 
See Remark~\ref{rem1}.
Summarizing, the approach proves to be robust in the energy norm as the physical parameter $f$ varies.

{\bf Robustness with respect to the position of the interface.} We conclude the series of the 2D tests with a set
of simulations aimed at checking that our approach is not sensitive
to the position of the interface with respect to the background mesh.
For this purpose, we vary the center of the circle that represents $\Gamma$:
\begin{equation}\label{eq:k}
\bc=(c_1,c_2), ~ c_1=\frac{h}{20}k \cos\left(\frac{k}{10} \pi \right), ~ c_2=\frac{h}{20}k \sin\left(\frac{k}{10} \pi\right), \quad k=1,2,...,20,
\end{equation}
where $h$ is the mesh size. We set $\mu_-=1$, $\mu_+=10$ and $f = 10$.

Just like the two previous sets of simulations, we consider the mesh with mesh size $h=1/64$ and the
$\bQ_{2} - Q_1$ pair.
Fig.~\ref{fig:variable_k} shows the $L^2$ error and weighted $H^1$ error \eqref{eq:weighted_H1}
for the velocity and weighted $L^2$ error \eqref{eq:weighted_L2} for the pressure against
the value of $k$ in \eqref{eq:k}. We see that all the errors
are fairly insensitive to the position of $\Gamma$ with respect to the background mesh,
indicating robustness.


\begin{figure}[h]
\centering
\includegraphics[width=.55\textwidth]{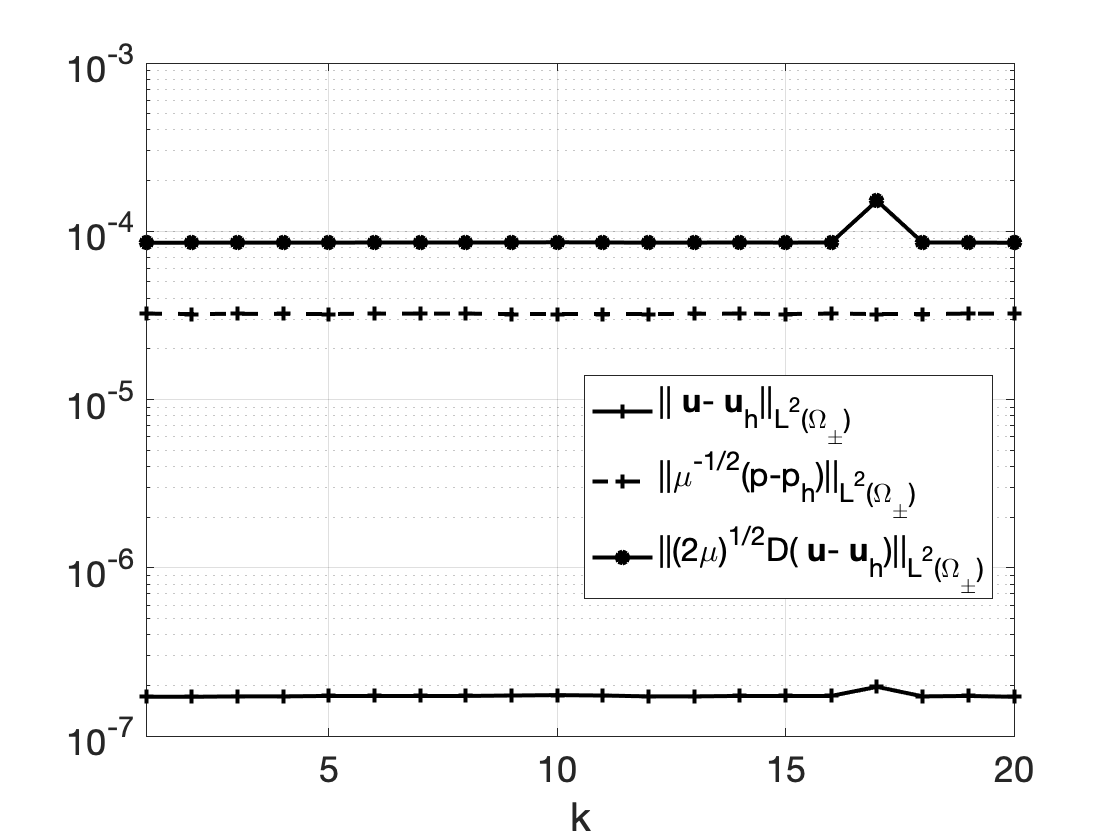}
  \caption{2D test with $\bc=\boldsymbol{0}$, $\mu_-=1$, $\mu_+=10$, and $f = 10$: $L^2$ error and weighted $H^1$ error \eqref{eq:weighted_H1} for the velocity and weighted $L^2$ error
  \eqref{eq:weighted_L2} for the pressure against the value of $k$ in \eqref{eq:k}.}
\label{fig:variable_k}
\end{figure}

\subsection{3D tests}
For the 3D tests, the domain $\Omega$ is cube $[-1.5,1.5]\times[-1.5,1.5]\times[-1.5,1.5]$
and interface $\Gamma$ is the unit sphere, centered at origin of the axes.
We characterize  $\Gamma$ as the zero level set of function $\phi(\bx)= || \bx ||^2_2 - 1$,
with $\bx = (x,y,z)$.
We consider the exact solution given by:
\begin{align}
p^+&=\frac{1}{2} x,  \hspace*{2.8cm} p^-=x, \label{ex_sol3D_p} \\
\bu^-&=g^-(x,y)
\left[
\begin{array}{c}
-y   \\
x \\
0
\end{array}
\right],
\qquad \bu^+=g^+(x,y)
\left[
\begin{array}{c}
-y   \\
x   \\
0
\end{array}
\right],  \label{ex_sol3D_u}
\end{align}
where
\begin{align*}
    &g^+(x,y)=\frac{1}{2\mu_+}(x^2+y^2+z^2),\\
    &g^-(x,y)=\frac{1}{2\mu_-}(x^2+y^2+z^2)+\frac{\mu_--2\mu_+\mu_--\mu_+}{2\mu_+\mu_-}.
\end{align*}
The forcing terms $\bbf^-$ and $\bbf^+$ are found by plugging the above solution in in \eqref{eq:Stokes1}.
We set $f=1$, $\mu_-=1$, and $\mu_+=100$. The surface tension coefficient is set
to $\sigma = -0.5x$.

Just like for the 2D tests, we impose a Dirichlet condition \eqref{eq:bcD} on the entire boundary, where function $\bg$
is found from $\bu^+$ in \eqref{ex_sol3D_u}.

To verify our implementation of the finite element method in Sec.~\ref{sec:FE} in three dimensions
and to further corroborate the results in Sec.~\ref{sec:error},
we consider structured meshes of tetrahedra with four levels of refinement.
The initial triangulation has mesh size $h = 1$ and all the other meshes are obtained by halving $h$
till $h = 0.125$.
All the meshes feature a local one-level refinement near the corners of $\Omega$.
We choose to use finite element pair $\bP_{2} - P_1$.
Fig.~\ref{fig:3Dsol} shows a visualization of the solution computed with mesh $h = 0.125$.
Fig.~\ref{fig:spatial_3D} shows the $L^2$ error and weighted $H^1$ error \eqref{eq:weighted_H1}
for the velocity and weighted $L^2$ error \eqref{eq:weighted_L2} for the pressure
against the mesh size $h$. For the small range of mesh sizes that we consider,
we observe almost cubic convergence in the $L^2$ norm for the velocity,
quadratic convergence in the weighted $L^2$ norm for the pressure and  in the weighted $H^1$ norm for the velocity.

\begin{figure}[h]
\centering
\includegraphics[width=.44\textwidth]{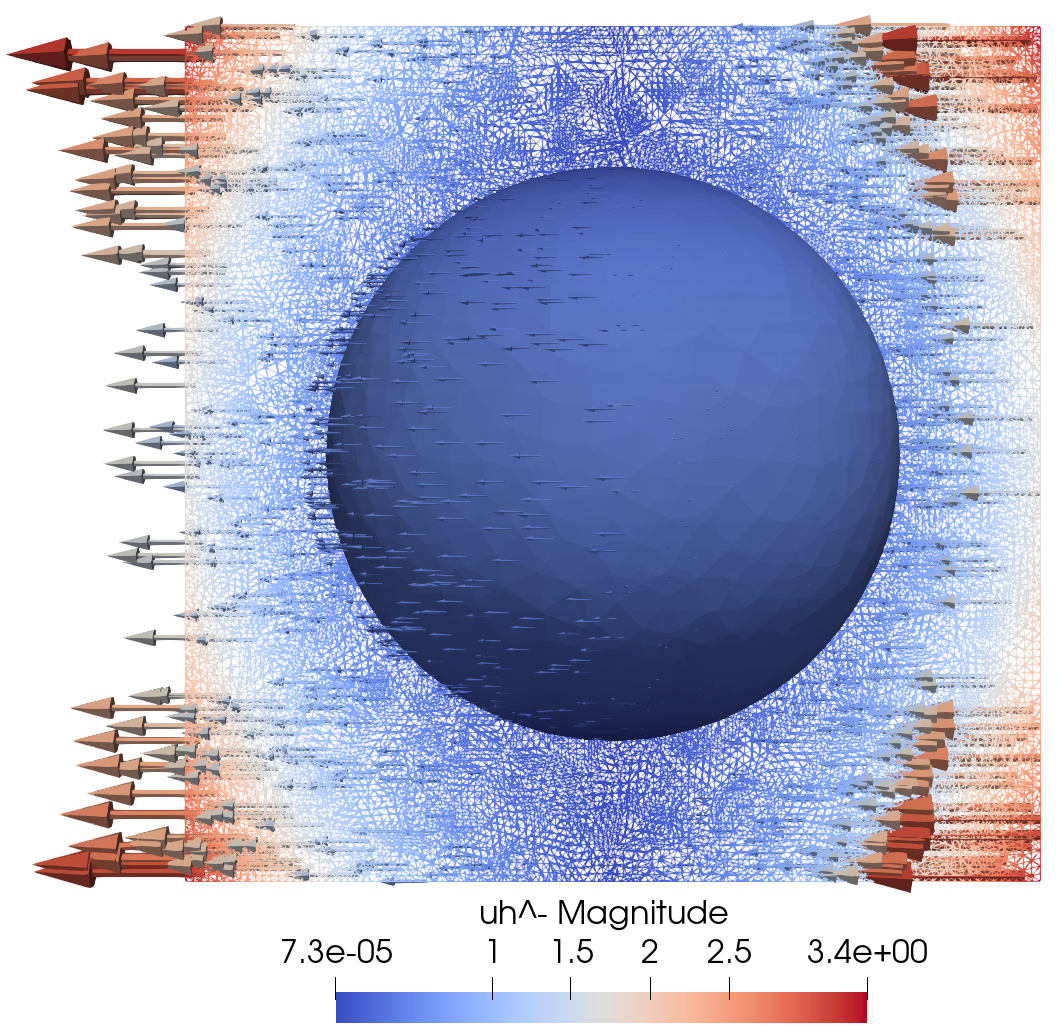}
\includegraphics[width=.39\textwidth]{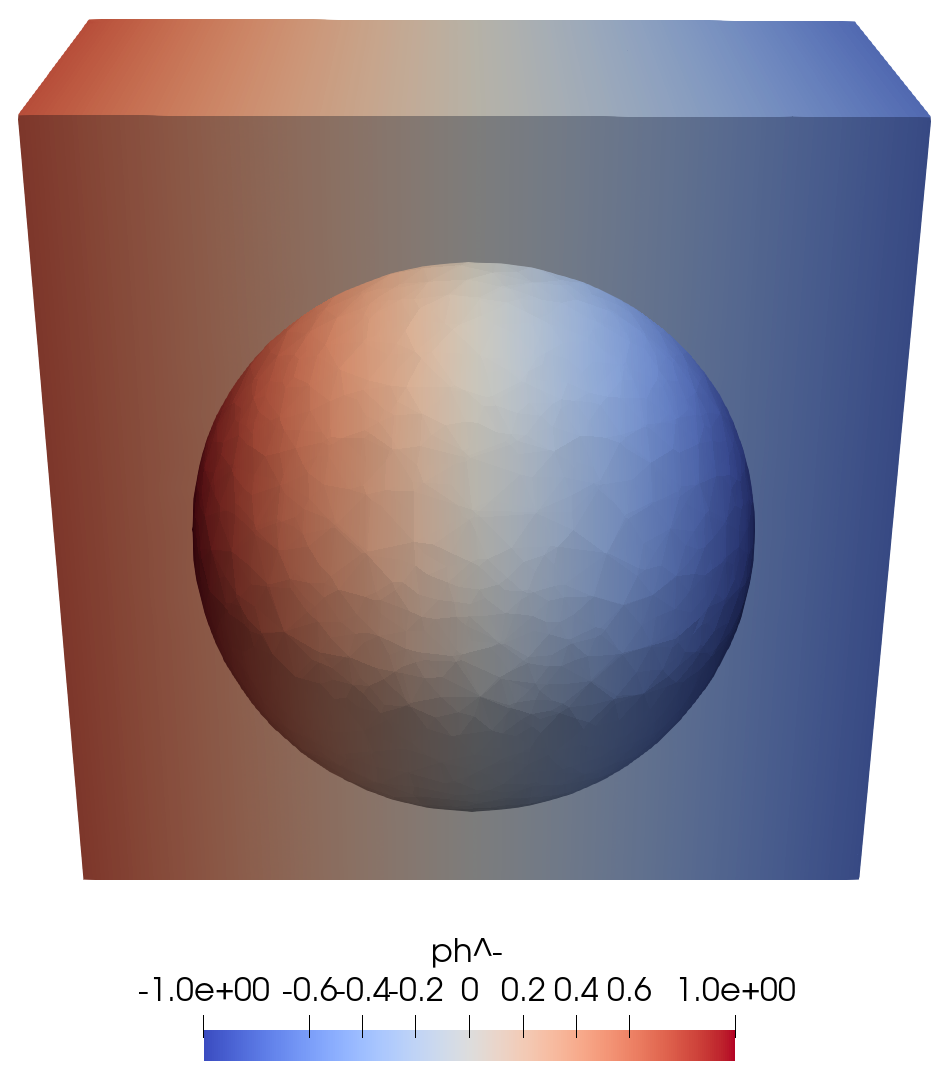}
  \caption{Approximation of exact solution \eqref{ex_sol3D_p}-\eqref{ex_sol3D_u} computed with the mesh with $h = 0.125$:
  velocity vectors colored with the velocity magnitude on the $xz$-section of $\Omega^+$ and in $\Omega^-$ (left) and pressure in
  $\Omega^-$ and half $\Omega^+$ (right).}
  \label{fig:3Dsol}
\end{figure}

\begin{figure}[h]
\centering
\includegraphics[width=.7\textwidth]{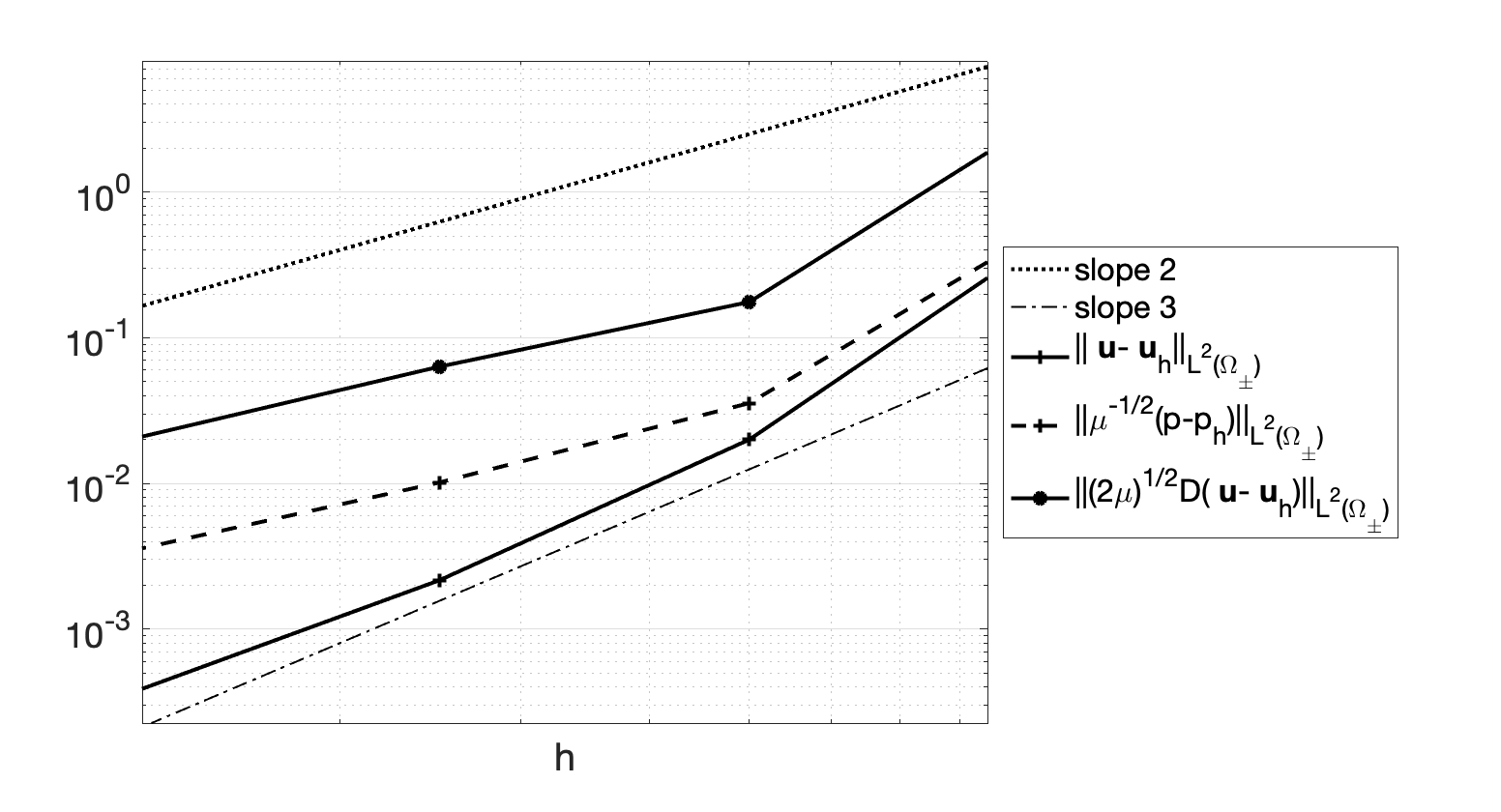}
  \caption{3D test: $L^2$ error and weighted $H^1$
  error \eqref{eq:weighted_H1} for the velocity and weighted $L^2$ error \eqref{eq:weighted_L2} for the pressure against the mesh size $h$.}
  \label{fig:spatial_3D}
\end{figure}


\section{Conclusions}\label{ref:concl}

In this paper, we focused on the two-phase Stokes problem with slip between phases,
which has received much less attention than its homogeneous counterpart (i.e.~no slip between the phases).
For the numerical approximation of this problem, we chose an isoparametric unfitted finite element approach
of the CutFEM or Nitsche-XFEM family. For the unfitted generalized Taylor--Hood finite element pair $\bP_{k+1}-P_k$,
we prove stability and optimal error estimates, which follow from an inf-sup stability property.
We show that the inf-sup stability constant is independent of the viscosity ratio, slip coefficient, 
position of the interface with respect to the background mesh and, of course, mesh size.

The 2D and 3D numerical experiments we used to test our approach feature an exact solution.
They have been designed to support the theoretical findings and demonstrate the robustness
of our approach for a wide range of physical parameter values. Finally, we show
that our unfitted approach is insensitive to the position of the interface between the two phases
with respect to the fixed computational mesh.


\section*{Acknowledgments}
This work was partially supported by US National Science Foundation (NSF) through grant DMS-1953535.
M.O.~also acknowledges the support from NSF through DMS-2011444.
A.Q.~also acknowledges the support from NSF through DMS-1620384.

\bibliographystyle{plain}
\bibliography{literatur}{}

\end{document}